\documentclass[12pt,reqno]{amsart}

\usepackage{amssymb}

\newtheorem{theorem}{Theorem}[section]
\newtheorem{proposition}[theorem]{Proposition}
\newtheorem{lemma}[theorem]{Lemma}
\newtheorem{corollary}[theorem]{Corollary}

\theoremstyle{definition}
\newtheorem{definition}[theorem]{Definition}

\numberwithin{equation}{section}

\newcommand{\CC}{\ensuremath{\mathbb{C}}}

\begin{document}

\baselineskip=15pt

\title[Holomorphic $G$-structure and foliated Cartan geometry]{Holomorphic $G$-structures and foliated Cartan 
geometries on compact complex manifolds}

\author[I. Biswas]{Indranil Biswas}

\address{School of Mathematics, Tata Institute of Fundamental
Research, Homi Bhabha Road, Mumbai 400005, India}

\email{indranil@math.tifr.res.in}

\author[S. Dumitrescu]{Sorin Dumitrescu}

\address{Universit\'e C\^ote d'Azur, CNRS, LJAD, France}

\email{dumitres@unice.fr}



\date{}
\begin{abstract}
This is a survey paper dealing with holomorphic $G$-structures and holomorphic Cartan geometries on compact complex 
manifolds. Our emphasis is on the foliated case: holomorphic foliations with transverse (branched or generalized) 
holomorphic Cartan geometries.
\end{abstract}

\maketitle

\tableofcontents

\section{Introduction}

In this survey paper we deal with holomorphic geometric structures on compact complex manifolds. We focus on 
holomorphic $G$-structures and on holomorphic Cartan geometries.

This subject of geometric structures has its roots in the study of the symmetry group of various geometric spaces. 
Important seminal work which shaped the subject was done by F. Klein and S. Lie. In particular, in his famous 
Erlangen address (1872), F. Klein established the program of a unifying background for all geometries (including 
the Euclidean, affine and projective geometries) as being that of homogeneous spaces $X\,=\,G/H$ under (the transitive) 
action of a symmetry group $G$; here $G$ is a finite dimensional connected Lie group, and $H$ is the stabilizer of a 
given point in the model space $X$.

The special case of the complex projective line $\CC \text{P}^1$ seen as a homogeneous space for the M\"obius group 
$\text{PSL}(2, {\mathbb C})$ appears naturally in the study of the second order linear differential equations on 
the complex domain. More precisely, the quotient of two local linearly independent solutions of an equation of the 
above type provides a local coordinate in $\CC \text{P}^1$ on which the monodromy of the equation (obtained by 
analytic continuation of the local solution along a closed curve avoiding the singularities) acts by elements in 
the M\"obius group $\text{PSL}(2, {\mathbb C})$. This procedure naturally leads to a {\it complex projective 
structure} on the Riemann surface $M$ bearing the second order linear differential equation, meaning that $M$ 
admits a holomorphic atlas with local coordinates in ${\CC}\text{P}^1$ such that the transition maps are
in $\text{PSL}(2, {\mathbb C})$.
 
Inspired by Fuchs' work on second order linear differential equations, Poincar\'e was the first one to understand 
that the above procedure provides a promising background which could lead to a general uniformization theorem for 
any Riemann surface $M$: one should find a second order linear differential equation on $M$ which is {\it 
uniformizing}, in the sense that, when pull-backed to the universal cover of $M$, the quotient of two local 
solutions provides a biholomorphic identification between the universal cover of $M$ and an open set $U$ in the 
complex projective line $\CC \text{P}^1$. The monodromy of the equation defined on $M$ furnishes a group 
homomorphism from the fundamental group of $M$ into the subgroup of the M\"obius group $\text{PSL}(2, {\mathbb 
C})$ preserving $U \,\subset\, \CC \text{P}^1$ and acting properly and discontinuously on $U$. The above identification 
(between the universal cover of $M$ and $U \,\subset\, \CC \text{P}^1$) is equivariant with respect to the monodromy 
morphism; this identification map is called the {\it developing map}. Consequently, the Riemann surface is 
uniformized as a quotient of $U \,\subset\, \CC \text{P}^1$ by a subgroup of $\text{PSL}(2, {\mathbb C})$ acting 
properly and discontinuously on $U$. Of course, in the generic (hyperbolic case) the open subset $U$ is the upper-half 
plane and the corresponding subgroup in the M\"obius group is a discrete subgroup of $\text{PSL}(2, \mathbb R)$. A 
complete proof based on the above considerations and also other proofs of the uniformization theorem for surfaces 
together with the historical background are presented in \cite{StG} (see also \cite{Gu}).
 
Another extremely important root of the subject is Elie Cartan's broad generalization of Klein's homogeneous model 
spaces to the corresponding infinitesimal notion, namely that of {\it Cartan geometries} (or {\it Cartan 
connections}) \cite{Sh}. Those geometric structures are infinitesimally modeled on homogeneous spaces $G/H$. A 
Cartan geometry on a manifold $M$ is naturally equipped with a curvature tensor which vanishes exactly when $M$ is 
locally modeled on $G/H$ in the following sense described by Ehresmann \cite{Eh}. In such a situation the Cartan 
geometry is called {\it flat}.

A manifold $M$ is said to be {\it locally modeled on a 
homogeneous space $G/H$} if $M$ admits an atlas with charts in $G/H$ satisfying the
condition that the transition maps are given 
by elements of $G$ using the left-translation action of $G$ on $G/H$. In this way $M$ is locally 
endowed with the $G/H$-geometry and all $G$-invariant geometrical features of $G/H$ have 
intrinsic meaning on $M$ \cite{Eh}.

In the same spirit as for complex projective structures on Riemann surfaces, homogeneous spaces are useful models 
for geometrization of topological manifolds of higher dimension. This was Thurston's point of view when he formulated
the geometrization conjecture for threefolds, using three-dimensional Riemannian homogeneous spaces $G/H$ (or 
equivalently, $H$ is compact).

Also Inoue, Kobayashi and Ochiai studied holomorphic projective connections on compact complex surfaces. A 
consequence of their work is that compact complex surfaces bearing holomorphic projective connections also admit 
a (flat) holomorphic projective structure such that the corresponding developing map
is injective. In particular, such complex 
surfaces are uniformized as quotients of open subsets $U$ in the complex projective plane $\CC \text{P}^2$ by a 
discrete subgroup in $\text{PSL}(3, \CC)$ acting properly and discontinuously on $U$ \cite{IKO, KO1,KO2}. Of 
course, most of the complex compact surfaces do not admit any holomorphic projective connection. It appears to be useful 
to allow more flexibility in the form of orbifold (or other mild) singularities.

In this spirit \cite{M1, M2} Mandelbaum introduced and studied {\it branched projective structures} on Riemann 
surfaces. A branched projective structure on a Riemann 
surface is given by some holomorphic atlas whose local charts are finite branched 
coverings of open subsets in ${\mathbb C}\text{P}^1$ while the transition maps lie in 
$\text{PSL}(2, {\mathbb C})$. 

Inspired by the above mentioned articles of Mandelbaum, we defined in \cite{BD} a more general notion of {\it 
branched holomorphic Cartan geometry} on a complex manifold $M$ which is valid also for complex manifolds in higher 
dimension and for non-flat Cartan geometries. We defined and studied a more general notion of {\it generalized 
holomorphic Cartan geometry} in \cite{BD5} (see also, \cite{AM}).

These notions of generalized and branched Cartan geometry are much more flexible than the usual one: branched 
holomorphic Cartan geometry are stable under pull-back through holomorphic ramified maps, while generalized 
holomorphic Cartan geometries are stable by pull-back through holomorphic maps. Also all compact complex 
projective manifolds admit branched flat holomorphic projective structure \cite{BD}.

A foliated version of a Cartan geometry (meaning Cartan geometry transverse to a foliation) was 
worked out in \cite{Bl, Mo}. We defined and studied the notion of a foliated holomorphic branched Cartan geometry 
(meaning a branched holomorphic Cartan geometry transverse to a holomorphic foliation) in \cite{BD3}. We define 
here also the foliated version of generalized Cartan geometry.

We show that these notions are rigid enough to enable one to obtain classification results.

The organization of this paper is as follows. After this introduction, Section 2 presents  the 
geometric notion of holomorphic $G$-structure and  a number of examples. Section 3 focuses on the special case of ${\rm GL}(2,\CC)$-structure 
and ${\rm SL}(2,\CC)$-structure and surveys some classification results. Sections 4 is about holomorphic Cartan 
geometry in the classical case and in the branched and generalized case. Section 5 deals with two equivalent 
definitions of foliated (branched, generalized) Cartan geometry and provides some classification results. At the 
end in Section 6 some related open problems are formulated.

\section{Holomorphic $G$-structure}\label{section G-struct}

Let $M$ be a complex manifold of complex dimension $n$. We denote by $TM$ (respectively, $T^*M$) its holomorphic 
tangent (respectively, cotangent) bundle. Let $R(M)$ be the holomorphic {\it frame bundle} associated to $TM$. We 
recall that for any point $m \,\in\, M$, the fiber $R(M)_m$ identifies with the set of $\CC-$linear maps from 
${\CC}^n$ to $T_mM$. The pre-composition with elements in $\rm{GL}(n, \CC)$ defines a (right) $\rm{GL}(n,
\CC)$--action on $R(M)$. This action is holomorphic, free and transitive in the fibers; therefore the quotient of $R(M)$ 
by $\rm{GL}(n, \CC)$ is the complex manifold $M$. The frame bundle is the typical example of a holomorphic 
principal $\rm{GL}(n, \CC)$--bundle over $M$.

Let $G$ be a complex Lie subgroup of $\rm{GL}(n, \CC)$. Then we have the following definition.

\begin{definition}
A holomorphic $G$-structure on $M$ is a holomorphic principal $G$-subbundle ${\mathcal R}\,\longrightarrow\,M$ in 
$R(M)$. Equivalently, ${\mathcal R}$ is a holomorphic reduction of the structure group of $R(M)$
to the subgroup $G\, \subset\, \rm{GL}(n, \CC)$.
\end{definition} 

It may be mentioned that the above definition is equivalent to the existence of a $\rm{GL}(n, \CC)$-equivariant 
holomorphic map $\Psi\,:\, R(M) \,\longrightarrow \,\rm{GL}(n, \CC)/G$. The pre-image, under $\Psi$, of the class of 
the identity element in $\rm{GL}(n, \CC)/G$ is the holomorphic $G$-subbundle in the above definition. Notice that 
the map $\Psi$ is a holomorphic section of the bundle $R(M)/G$ (with fiber type $\rm{GL}(n, \CC)/G$), where the $G$ 
action on $R(M)$ is the restriction of the principal $\rm{GL}(n, \CC)$-action.

The $G$-structure $\mathcal R\,\subset\, R(M)$ is said to be {\it flat} (or, {\it integrable}) if for every $m \,\in\, M$, 
there exists an open neighborhood $U$ of $0\,\in\, \CC^n$ and a holomorphic local biholomorphism $$f \,:\, (U,\,0) 
\,\longrightarrow \,(M,\,m)$$ such that the image of the differential $df$ of $f$, lifted as a map $df \,:\, R(U) 
\,\longrightarrow\, R(M)$, satisfies $df(s(U)) \,\subset\, \mathcal R$, where $s\, :\, U\, \longrightarrow\,
R(U)$ is the section of the projection $R(U)\, \longrightarrow\, U$ given by the standard frame of $\CC^n$.

For more details about this concept the reader is referred to \cite{Kob}. Let us now describe some important 
examples which are at the origin of this concept.

{\bf ${\rm SL}(n,\CC)$-structure}.\, An ${\rm SL}(n, \CC)$-structure on $M$ is equivalent with giving a non-vanishing 
holomorphic section of the canonical bundle $K_M\,=\,\bigwedge^n T^*M$. This holomorphic section $\omega \,\in\, 
H^0(M,\, K_M)$ defines the holomorphic principal $SL(n, \CC)$-subbundle in $R(M)$ with fiber above $m\,\in\, M$ 
given by the subset ${\mathcal R}_m \,\subset\, R(M)_m$ consisting of all frames $l \,:\, \CC^n\,\longrightarrow\, 
T_mM$ such that $l^*\omega_m$ is the $n$-form on $\CC^n$ given by the determinant. Therefore,
$\mathcal R\,\subset\, R(M)$ is defined by those frames which have volume $1$ with respect to $\omega$.

Recall that with respect to holomorphic local coordinates $(z_1,\, \ldots,\, z_n)$ on $M$, a holomorphic section 
$\omega \,\in\, H^0(M,\, K_M)$ is given by an expression $f(z_1, \ldots,z_n)dz_1 \wedge dz_2\wedge 
\ldots\wedge dz_n$, with $f$ a non-vanishing local holomorphic function. We can always operate a biholomorphic 
changing of local coordinates in order to have $f$ identically equal to $1$. This implies that an $SL(n, 
\CC)$-structure is always flat.

{\bf ${\rm Sp}(2n,\CC)$-structure}.\, This structure is equivalent with giving a holomorphic pointwise non-degenerate 
holomorphic $2$-form on $M$. Such a section defines an isomorphism between $TM$ and $T^*M$ and implies $M$ has even 
dimension $2n$. Equivalently, an ${\rm Sp}(2n,\CC)$-structure is a section $\omega_0 \,\in\,
H^0(M,\, \bigwedge ^2T^*M)$ such that $\omega_0^n$ is a nowhere vanishing section of
$\bigwedge^{2n}T^*M$. The corresponding ${\rm Sp}(2n,\CC)$-structure is 
flat if there exists holomorphic local coordinates with respect to which the local expression of $\omega_0$ is $dz_1 
\wedge dz_2 + \ldots + dz_{2n-1} \wedge dz_{2n}$. By Darboux theorem this is equivalent
to the condition that $d \omega_0\,=\,0$, i.e., 
$\omega_0$ is a holomorphic symplectic form (or equivalently, a holomorphic non-degenerate closed $2$-form).

If $M$ is a compact K\"ahler manifold, then any holomorphic form on $M$ is automatically closed. In this 
case any holomorphic ${\rm Sp}(2n,\CC)$-structure on $M$ is flat: it endows $M$ with a holomorphic symplectic form.

Holomorphic symplectic structures on compact complex non K\"ahler manifolds were constructed in \cite{Gu1,Gu2}.

{\bf ${\rm O}(n,\CC)$-structure.}\, Giving an ${\rm O}(n, \CC)$-structure on $M$ is equivalent to giving a pointwise 
nondegenerate (i.e., of rank $n$) holomorphic section $g$ of the bundle $S^2(T^*M)$ of symmetric complex quadratic 
forms in the holomorphic tangent bundle of $M$. This structure is also known as a {\it holomorphic Riemannian metric} 
\cite{Le,Gh,Du,DZ,BD3}. Notice that the complexification of any real-analytic Riemannian or pseudo-Riemannian metric 
defines a local holomorphic Riemannian metric which coincides with the initial (pseudo-)Riemannian metric on the 
real locus.

The fiber above $m\,\in\, M$ of the corresponding subbundle $\mathcal R_m\,\subset\, R(M)_m$ is formed by all frames $l 
\,:\, \CC^n\,\longrightarrow \, T_mM$ such that $l^*(g_m)\,=\,dz_1^2+ \ldots +dz_n^2$. In local holomorphic coordinates 
$(z_1,\, \ldots,\, z_n)$ on $M$ the local expression of the section $g \,\in\, H^0(M,\, S^2(T^{*}M))$ is 
$\sum_{ij}g_{ij}dz_idz_j,$ with $g_{ij}$ a nondegenerate $n\times n$ matrix with holomorphic coefficients. There 
exists local coordinates with respect to which the matrix $g_{ij}$ is locally constant if and only if the holomorphic 
Riemannian metric is flat (e.g., has vanishing curvature tensor). Hence the ${\rm O}(n,\CC)$-structure is flat if and only 
if the corresponding holomorphic Riemannian metric is flat (i.e, there exists local holomorphic coordinates with 
respect to which the local expression of $g$ is $dz_1^2+\ldots +dz_n^2$).

{\bf Homogeneous tensors.}\, The three previous examples are special cases of the following construction. Consider 
a holomorphic linear $\rm{GL}(n, \CC)$--action on a finite dimension complex vector space $V$ given by 
$i\,:\,\rm{GL}(n, \CC)\,\longrightarrow\, \rm{GL}(V)$. This action together with $R(M)$ define a holomorphic vector 
bundle with fiber type $V$ by the usual quotient construction: two elements $(l,\,v),\, (l',\,v')\,\in\, R(M)\times 
V$ are equivalent if there exists $g \,\in \,\rm{GL}(n, \CC)$ such that $$(l',\, v')\,=\,(l \cdot g^{-1}, \,i(g) 
\cdot v).$$ The holomorphic vector bundle constructed this way will be denoted $R(M) \ltimes_i V$.

Note that for the linear representation $$V^{p,q}\,=\, (\CC^n)^{\otimes p} \otimes (({\CC}^n)^*)^{\otimes q},$$ 
where $p,\, q$ are nonnegative integers, the corresponding vector bundle $R(M) \ltimes V$ coincides with the bundle 
of holomorphic tensor $TM^{\otimes p} \otimes (TM^*)^{\otimes q}$. Moreover any irreducible algebraic representation 
$i \,:\, \rm{GL}(n, \CC) \,\longrightarrow\, \rm{GL}(V)$ is known to be a factor of the representation 
$(\CC^n)^{\otimes p} \otimes (({\CC}^n)^*)^{\otimes q}$ for some integers $(p,\,q)$,

Now consider a holomorphic section of $TM^{\otimes p} \otimes (TM^*)^{\otimes q}$ (namely, a holomorphic tensor of 
type $(p,q)$) given by a $\rm{GL}(n, \CC)-$equivariant map $$\Theta\,\,:\,\, R(M)\,\longrightarrow\, V^{p,q}.$$
Assume that the 
image of $\Theta$ lies in a $\rm{GL}(n, \CC)-$orbit of $V^{p,q}$, with stabilizer $G_0 \,\subset\, \rm{GL}(n, \CC)$. 
Such a tensor is called homogeneous (of order $0$). Under this assumption, the map $\Theta \,:\, R(M)\,\longrightarrow
\,\rm{GL}(n, \CC)/G_0$ defines a holomorphic $G_0$-structure on $M$. Conversely, any $G_0$-structure comes in the 
above way from a unique (up to a scalar constant) holomorphic tensor $\Theta$.

{\bf ${\rm CO}(n,\CC)$-structure}.\, Here ${\rm CO}(n, \CC) \,\subset\, \rm{GL}(n, \CC)$ is the subgroup generated 
by the homotheties and ${\rm O}(n, \CC)$; it is known as the {\it orthogonal similitude group}. It is the stabilizer of the 
line generated by the standard holomorphic Riemannian metric $dz_1^2+\ldots +dz_n^2$ in the linear representation of 
$\rm{GL}(n, \CC)$ on $S^2((\CC^n)^{*})$. Therefore a ${\rm CO}(n, \CC$)--structure is {\it a holomorphic conformal 
structure} on $M$: it is given by a holomorphic line $\mathcal L\,\subset\, S^2(T^*M)$ such that any nonvanishing 
local holomorphic section of $\mathcal L$ is pointwise a nondegenerate complex quadratic form on the holomorphic 
tangent bundle. A holomorphic conformal structure is flat if and only if a nontrivial holomorphic local section of 
$\mathcal L$ can be expressed in local holomorphic coordinates as $f(z_1, \ldots, z_n) (dz_1^2+ \ldots +dz_n^2),$ 
with $f$ a holomorphic function. By a famous result of Gauss, this is always the case when the complex dimension $n$ 
equals $2$. In higher dimension the flatness of the holomorphic conformal structures is equivalent to the 
vanishing of a certain curvature tensor. The curvature tensor in question is known as the Weyl conformal tensor.

{\bf ${\rm CSp}(2n,\CC)$-structure}.\, The subgroup ${\rm CSp}(2n, \CC) \,\subset\, \rm{GL}(2n, \CC)$ is generated 
by ${\rm Sp}(2n, \CC)$ and the homotheties; it is known as the {\it symplectic similitude group}. This ${\rm 
CSp}(2n, \CC)$ is the stabilizer of the line generated by the standard symplectic form $dz_1 \wedge dz_2 + \ldots + 
dz_{2n-1} \wedge dz_{2n}$ in the linear representation of $\rm{GL}(n, \CC)$ on $\bigwedge^2((\CC^n)^{*})$. On a 
complex manifold $M$, a ${\rm CSp}(2n,\CC)$--structure is a holomorphic line subbundle $\mathcal L \,\subset\, 
\bigwedge^2(TM^{*})$ such that any nonvanishing local holomorphic section is pointwise nondegenerate. This structure 
is called a {\it a nondegenerate holomorphic twisted $2$-form}. It is flat exactly when a nontrivial holomorphic 
local section is locally expressed as $f(z_1, \ldots, z_n) (dz_1 \wedge dz_2 + \ldots + dz_{2n-1} \wedge dz_{2n})$, 
with $f$ a local holomorphic function. Moreover, $f$ can be chosen to be constant if and only if the local section 
is closed. In this case we say that $M$ is endowed with a {\it twisted holomorphic symplectic structure}.

{\bf ${\rm GL}(2,\CC)$-structure}.\, Here the corresponding group homomorphism $$\text{GL}(2,{\mathbb C}) 
\,\longrightarrow\, \text{GL}(n,{\mathbb C})$$ is defined by the $(n-1)$-th symmetric product $S^{(n-1)} (\CC^{2})$ 
of the standard representation of $\text{GL}(2,{\mathbb C})$. Recall that $S^{(n-1)} (\CC^{2})$ gives the unique, 
up to tensoring with a $1$-dimensional representation, irreducible $\text{GL}(2,{\mathbb C})$-linear representation 
in complex dimension $n$. This representation is the induced $\text{GL}(2,{\mathbb C})$--action on the homogeneous 
polynomials of degree $(n-1)$ in two variables.

Therefore giving a holomorphic $\text{GL}(2,{\mathbb C})$-structure on a complex manifold $M$ of complex dimension 
$n\, \geq\, 2$ is equivalent to giving a holomorphic rank two vector bundle $E$ over $M$ together with a 
holomorphic isomorphism $TM \,\stackrel{\simeq}{\longrightarrow}\, S^{n-1} (E)$, where $S^{n-1}(E)$ is the 
$(n-1)$-th symmetric power of $E$.

{\bf ${\rm SL}(2,\CC)$-structure}.\, In this case the group homomorphism $\text{SL}(2,{\mathbb C}) 
\,\longrightarrow\, \text{GL}(n,{\mathbb C})$ is defined by the $(n-1)$-th symmetric product $S^{(n-1)} (\CC^{2})$ 
of the standard representation of $\text{SL}(2,{\mathbb C})$. This is the restriction to $\text{SL}(2,{\mathbb C}) 
\,\subset \,\text{GL}(2,{\mathbb C})$ of the homomorphism in the previous example.

An $\text{SL}(2,{\mathbb C})$-structure on a complex manifold $M$ of complex dimension $n\, \geq\, 2$ is the same 
data as a holomorphic rank two vector bundle $E$ with trivial determinant over $M$ together with a holomorphic 
isomorphism $TM \,\stackrel{\simeq}{\longrightarrow}\, S^{n-1}(E)$.

\section{\text{GL(2)}-geometry and \text{SL(2)}-geometry}

In this section we study the holomorphic $\text{GL}(2,{\mathbb C})$-structures (also called 
$\text{GL(2)}$-geometry) and the holomorphic $\text{SL}(2,{\mathbb C})$-structures (also called 
\text{SL(2)}-geometry) on complex manifolds.

An shown before, holomorphic $\text{GL}(2,{\mathbb C})$-structures and $\text{SL}(2, {\mathbb C})$-structures are 
in fact particular cases of holomorphic irreducible reductive $G$-structures \cite{Kob, HM}. Recall that they 
correspond to the holomorphic reduction of the structure group of the frame bundle $R(M)$ of the manifold $M$ from 
$\text{GL}(n,{\mathbb C})$ to $\text{GL}(2,{\mathbb C})$ and $\text{SL}(2,{\mathbb C})$ 
respectively. Also, recall that for a $\text{GL}(2,{\mathbb C})$-structure, the 
corresponding group homomorphism $\text{GL}(2,{\mathbb C}) \,\longrightarrow\, 
\text{GL}(n,{\mathbb C})$ is given by the $(n-1)$-th symmetric product of the standard
representation of $\text{GL}(2,{\mathbb C})$. This 
$n$-dimensional irreducible linear representation of $\text{GL}(2,{\mathbb C})$ is also given by the 
induced action on the homogeneous polynomials of degree $(n-1)$ in two variables. For an
$\text{SL}(2,{\mathbb C})$-geometry, the corresponding homomorphism $\text{SL}(2,{\mathbb C}) 
\,\longrightarrow\, \text{GL}(n,{\mathbb C})$ is the restriction of the above homomorphism to
$\text{SL}(2,{\mathbb C})\, \subset\,\text{GL}(2,{\mathbb C})$.

\begin{proposition}
Let $M$ be a complex manifold endowed with a holomorphic ${\rm SL}(2,{\mathbb 
C})$-structure (respectively, a holomorphic ${\rm GL}(2,{\mathbb C})$-structure). Then the following hold:
\begin{enumerate}
\item[(i)] If the complex dimension of $M$ is odd, then $M$ is endowed with a holomorphic Riemannian metric 
(respectively, a holomorphic conformal structure).

\item[(ii)] If the complex dimension of $M$ is even, then $M$ is endowed with a nondegenerate holomorphic 
$2$-form (respectively, a nondegenerate holomorphic twisted $2$-form). Moreover, if $M$ is K\"ahler then $M$ is 
endowed with a holomorphic symplectic form (respectively, a holomorphic twisted symplectic form).
\end{enumerate}
\end{proposition} 

\begin{proof}
A holomorphic $\text{SL}(2,{\mathbb C})$-structure on a complex surface $M$ is a holomorphic trivialization of 
the canonical bundle $K_M\,=\, \bigwedge^2 T^*M$. Indeed, the standard action of $\text{SL}(2,{\mathbb C})$ on 
${\mathbb C}^2$ preserves the holomorphic volume form $dz_1\bigwedge dz_2$.

Notice that in the case of complex dimension two the holomorphic volume form $dz_1\bigwedge dz_2$ coincides with 
the standard holomorphic symplectic structure. Hence the (flat) ${\rm SL}(2,\CC)$--structure coincide with the 
(flat) ${\rm Sp}(2,\CC)$--structure.

This produces a nondegenerate complex quadratic form on $S^{n-1}({\mathbb C}^2)$, for any $n-1$ even. Also, for any 
odd integer $n-1$, this endows $S^{n-1}({\mathbb C}^2)$ with a nondegenerate complex alternating 2-form on the 
symmetric product $S^{n-1}({\mathbb C}^2)$. Consequently, the corresponding irreducible linear representation 
$\text{SL}(2,{\mathbb C})\,\longrightarrow\, \text{GL}(n,{\mathbb C})$ preserves a nondegenerate complex quadratic 
form on ${\mathbb C}^n$, if $n$ is odd and a nondegenerate 2-form on ${\mathbb C}^n$, if $n$ is even (see, for 
example, Proposition 3.2 and Sections 2 and 3 in \cite{DG}; see also \cite{Kr}).

The above linear representation $\text{GL}(2,{\mathbb C})\,\longrightarrow\, \text{GL}(n,{\mathbb C})$ preserves 
the line in ${(\mathbb C}^n)^*\otimes {(\mathbb C}^n)^*$ spanned by the above tensor. The action of 
$\text{GL}(2,{\mathbb C})$ on this line is nontrivial.

(i) Assume that the complex dimension of $M$ is odd. The above observations imply that the $\text{SL}(2,{\mathbb 
C})$-geometry (respectively, $\text{GL}(2,{\mathbb C})$-geometry) induces a {\it holomorphic Riemannian metric} 
(respectively, a holomorphic conformal structure) on M.

(ii) Assume that the complex dimension of $M$ is even. The $\text{SL}(2,{\mathbb C})$-geometry (respectively, 
$\text{GL}(2,{\mathbb C})$-geometry) induces a holomorphic nondegenerate $2$- form (respectively, a holomorphic 
twisted nondegenerate $2$-form).

Moreover, if $M$ is K\"ahler, any holomorphic differential form on $M$ is closed. This implies that
the holomorphic nondegenerate $2$-form is a holomorphic symplectic form.
\end{proof} 

The simplest nontrivial examples of $\text{GL}(2,{\mathbb C})$ and $\text{SL}(2,{\mathbb C})$ structures are 
provided by the complex threefolds. In this case we have the following:

\begin{proposition}
Let $M$ be a complex threefold.
\begin{enumerate}
\item[(i)] A holomorphic ${\rm SL}(2,{\mathbb C})$-structure on $M$ is a holomorphic Riemannian metric.

\item[(ii)] A holomorphic ${\rm GL}(2,{\mathbb C})$-structure on $M$ is a holomorphic conformal structure.
\end{enumerate}
\end{proposition}

\begin{proof} 
(i) Recall that $S^{2}({\mathbb C}^2)$ is the unique irreducible $\text{SL}(2,{\mathbb C})$-representation on
the three-dimensional complex vector space. It coincides with the natural action of $\text{SL}(2,{\mathbb C})$ on 
the vector space of homogeneous quadratic polynomials in two variables
\begin{equation}\label{hp}
\{aX^2+bXY +c Y^2\, \mid\, a,\, b,\, c\, \in\, \mathbb C\}.
\end{equation}
The above (coordinate changing) action preserves the discriminant 
\begin{equation}\label{ded}
\Delta\,:=\,b^2-4ac
\end{equation}
which is a nondegenerate complex quadratic form. Notice also that the action of $-{\rm Id}\,\in\, 
\text{SL}(2,{\mathbb C})$ is trivial. Consequently, we obtain a holomorphic isomorphism between 
$\text{PSL}(2,{\mathbb C})$ and the complex orthogonal group $\text{SO}(3, \mathbb C)$, the connected component of 
the identity in the complex orthogonal group $\text{O}(3, \mathbb C)$.

The discriminant $\Delta$ in \eqref{ded} induces a holomorphic Riemannian metric on the complex threefold $M$. This 
holomorphic Riemannian metric is given by a reduction of the structural group of the frame bundle $R(M)$ to the 
orthogonal group $\text{O}(3, \mathbb C)$. On a double unramified cover of $M$ there is a reduction of the 
structural group of $R(M)$ to the connected component of the identity, namely the subgroup $\text{SO}(3, \mathbb 
C)$.

An $\text{SL}(2,{\mathbb C})$-geometry on a complex threefold $M$ is exactly the same data as a holomorphic 
Riemannian metric on $M$. Moreover, a $\text{PSL}(2,{\mathbb C})$-geometry is the same data as a holomorphic 
Riemannian metric with a compatible holomorphic orientation (i.e., a nowhere vanishing holomorphic section $vol$ 
of $\bigwedge^3T^*M$ such that any pointwise basis of $TM$ which is orthogonal with respect to the holomorphic 
Riemannian metric has volume one).

(ii) Notice that the $\text{GL}(2,{\mathbb C})$-representation (by coordinate changing) on the vector space defined 
in \eqref{hp} globally preserves the line generated by the discriminant $\Delta$ in \eqref{ded}(and acts nontrivially 
on it). Recall that the action of -Id is trivial.

This gives an isomorphism 
between $\text{GL}(2,{\mathbb C})/({\mathbb Z}/2{\mathbb Z})$ and the conformal group $$\text{CO}(3, \mathbb C) 
\,=\,(\text{O}(3, \mathbb C) \times {\mathbb C}^* )/({\mathbb Z}/2{\mathbb Z}) \,=\, \text{SO}(3, \mathbb C) \times {\mathbb C}^*.$$ Consequently, 
the $\text{GL}(2,{\mathbb C})$--structure coincides with a holomorphic reduction of the structure 
group of the frame bundle $R(M)$ to $\text{CO}(3, \mathbb C)$. This holomorphic reduction of the 
structure group defines a holomorphic conformal structure on $M$. Notice that $\text{CO}(3, \mathbb C)$ being connected, the two different holomorphic orientations
of a three dimensional holomorphic Riemannian manifold are conformally equivalent.
\end{proof} 

\subsection{Models of the flat holomorphic conformal geometry}
 
Recall that flat conformal structures in complex dimension $n \,\geq\, 3$ are locally modeled on the 
quadric
$$
Q_n\,\, :=\,\, \{[Z_0 : Z_1 : \cdots : Z_{n+1}] \, \mid\, Z_0^2+Z_1^2+ \ldots +Z_{n+1}^2\,=\,0\}
\, \subset\, {\mathbb C}{\mathbb P}^{n+1}\, .
$$
The holomorphic automorphism group of $Q_n$ is $\text{PSO}(n+2, \mathbb C)$.

Let us mention that $Q_n$ is identified with the real Grassmannian of oriented $2$--planes in ${\mathbb R}^{n+2}$ 
(see, for instance, Section 1 in \cite{JR2}). As a real homogeneous space we have the following identification: 
$$Q_n\,=\, {\rm SO}(n+2, \mathbb R)/({\rm SO}(2, \mathbb R)\times{\rm SO}(n, \mathbb R)).$$

Moreover, the standard action of ${\rm SO}(n+2, \mathbb R)$ on $Q_n$ is via holomorphic automorphisms. To see this 
let us describe the complex structure on $Q_n$ from the (real) Lie group theoretic point of view.

Note that the 
real tangent bundle of $Q_n\,=\, {\rm SO}(n+2, \mathbb R)/({\rm SO}(2, \mathbb R)\times
{\rm SO}(n, \mathbb R))$ is identified with the  vector bundle associated to the principal
${\rm SO}(2, \mathbb R)\times{\rm SO}(n, \mathbb R)$--bundle
$$
{\rm SO}(n+2, \mathbb R)\, \longrightarrow\, {\rm SO}(n+2, \mathbb R)/({\rm SO}(2, \mathbb R)\times
{\rm SO}(n, \mathbb R))
$$
for the adjoint action of ${\rm SO}(2, \mathbb R)\times{\rm SO}(n, \mathbb R)$ on the 
quotient of real Lie algebras ${\rm so}(n+2, \mathbb R)/({\rm so}(2, \mathbb R)
\oplus{\rm so}(n, \mathbb R))$.
The complex structure of the tangent space ${\rm so}(n+2, \mathbb R)/({\rm so}(2, \mathbb R)
\oplus{\rm so}(n, \mathbb R))$ at the point $({\rm SO}(2, \mathbb R)\times{\rm SO}(n, \mathbb R))/
({\rm SO}(2, \mathbb R)\times{\rm SO}(n, \mathbb R))$ of $Q_n$ is given 
by the almost-complex operator $J$ such that its exponential $\{\exp(tJ)\}_{t\in\mathbb R}$
is the adjoint action of the factor ${\rm SO}(2, \mathbb R)$. Since ${\rm SO}(2, \mathbb R)$
lies in the center of ${\rm SO}(2, \mathbb R)\times{\rm SO}(n, \mathbb R)$, this almost complex structure $J$ is
preserved by the action of the adjoint action of ${\rm SO}(2, \mathbb R)\times{\rm SO}(n, \mathbb R)$
on ${\rm so}(n+2, \mathbb R)/({\rm so}(2, \mathbb R)
\oplus{\rm so}(n, \mathbb R))$. Consequently, translating this almost complex structure $J$ by the
action of ${\rm SO}(n+2, \mathbb R)$ we get an almost complex structure on $Q_n$. This
almost complex structure is integrable.

The quadric $Q_n$  is an irreducible Hermitian symmetric space of type \textbf{BD I} 
\cite[p.~312]{Bes}. For more about its geometry and that of its noncompact dual $D_n$ (as a Hermitian symmetric 
space) the reader is referred to \cite[Section 1]{JR2}. A detailed description of $D_3$ and $Q_3$ will be given in Section \ref{Q_3}.

Let $M$ be a compact complex manifold of complex dimension $n$ endowed with a flat holomorphic conformal structure. 
Then the pull-back of the flat holomorphic conformal structure to the universal cover $\widetilde{M}$ of $M$ gives 
rise to a local biholomorphism $\text{dev} \,:\, \widetilde{M}\,\longrightarrow\, Q_n$ which is called the {\it 
developing map} and also gives rise to a group homomorphism from the fundamental group of $M$ to $\text{PSO}(n+2, 
\mathbb C)$ which is called the {\it monodromy homomorphism}. The developing map is uniquely defined up to the 
post-composition by an element of $\text{PSO}(n+2, \mathbb C)$, and the monodromy morphism is well-defined up to an 
inner automorphism of $\text{PSO}(n+2, \mathbb C)$. Moreover, the developing map is equivariant with respect to the 
action of the fundamental group of $M$ on $\widetilde{M}$ (by deck transformations) and on $Q_n$ (through the 
monodromy homomorphism).

For more details about this classical description of flat geometric structures, due to C. Ehresmann, the reader is 
referred to \cite{Eh,Sh} (see also Section \ref{sect-Cartan}). This also leads to the following characterizations 
of the quadric $Q_n$.

\begin{proposition}[{\cite{Oc,HM,Ye}}]\label{quadrics} Let $M$ be a compact complex simply connected manifold of 
complex dimension $n$, endowed with a flat holomorphic conformal structure. Then $M$ is biholomorphic with $Q_n$ and 
the conformal structure on it is the standard one.
\end{proposition}

\begin{proof} Since $M$ is compact and simply connected, the developing map $\text{dev} \,:\, \widetilde{M}\,=\,M 
\,\longrightarrow\, Q_n$ is a biholomorphism. Moreover, this developing map intertwines the flat holomorphic 
conformal structure on $M$ with the standard conformal structure of $Q_n$.
\end{proof}

All quadrics $Q_n$ are {\it Fano manifolds}. Recall that a Fano manifold is a compact complex projective manifold 
$M$ such that its anticanonical line bundle $K_M^{-1}$ is ample. Fano manifolds are intensively studied. In 
particular, it is known that they are rationally connected \cite{Ca2}, \cite{KMM}, and simply connected \cite{Ca1}.

An important invariant for Fano manifolds is their {\it index}. By definition, the index of a Fano manifold $M$ is 
the maximal positive integer $l$ such that the canonical line bundle $K_M$ is divisible by $l$ in the Picard group 
of $M$ (i.e., the group of holomorphic line bundles on $M$). In other words, the index $l$ is the maximal positive 
integer such that there exists a holomorphic line bundle $L$ over $M$ such that $L^{\otimes l}=K_M$.

Theorem \ref{Fano} below, proved in \cite{BD1}, shows, in particular, that among the quadrics $Q_n$,\, $n \,\geq\, 3$, 
only $Q_3$ admits a holomorphic $\text{GL}(2,{\mathbb C})$--structure.

\begin{theorem}[{\cite{BD1}}]\label{Fano}
Let $M$ be a Fano manifold, of complex dimension $n \,\geq\, 3$, that admits a holomorphic 
$\text{GL}(2,{\mathbb C})$--structure. Then $n\,=\,3$, and $M$ is biholomorphic to the quadric $Q_3$ (the 
${\rm GL}(2,{\mathbb C})$--structure being the standard one).
\end{theorem} 

\begin{proof}
Let
$$TM \,\stackrel{\sim}{\longrightarrow}\, S^{n-1}(E)$$
be a holomorphic $\text{GL}(2,{\mathbb C})$--structure on $M$, where $E$ is a holomorphic vector bundle
of rank two on $X$. A direct computation shows that
$$K_M\,= \, (\bigwedge\nolimits^2 E^*)^{\frac{n(n-1)}{2}}\, .$$
Hence the index of $M$ is at least $\frac{n(n-1)}{2}$.
It is a known fact that the index of a Fano manifold $N$ of complex dimension $n$ is at most $n+1$. Moreover, the 
index is maximal ($=\, n+1$) if and only if $N$ is biholomorphic to the projective space 
${\mathbb C}{\mathbb P}^n$, and the index equals $n$ if and only if $N$ is biholomorphic to the quadric 
\cite{KO1}. These imply that
$n\,=\,3$ and $M$ is biholomorphic to the quadric $Q_3$. 

The $\text{GL}(2,{\mathbb C})$--structure on the quadric $Q_3$ must be flat \cite{KO, Kl, HM,Ye}.
Since the quadric is simply connected this flat $\text{GL}(2,{\mathbb C})$--structure coincides with the 
standard one \cite{Oc,HM,Ye} (see also Proposition \ref{quadrics}).
\end{proof} 

Recall that a general result of Borel on Hermitian irreducible symmetric spaces shows 
that the noncompact dual is always realized as an open subset of its compact dual.
 
We will give below a geometric description of the noncompact dual $D_3$ of $Q_3$ as an open 
subset in $Q_3$ which seems to be less known (it was explained to us by Charles Boubel whom we warmly acknowledge).

\subsection{Geometry of the quadric $Q_3$}.   \label{Q_3}
 
Consider the complex quadric form $q_{3,2}\,:=\,Z_0^2 +Z_1^2+Z_2^2-Z_3^2-Z_4^2$ of five variables,
and let $$Q\,\,\subset\,\, {\mathbb C}{\mathbb P}^4$$
be the quadric defined 
by the equation $q_{3,2}\,=\,0$. Then $Q$ is biholomorphic to $Q_3$.
Let ${\rm O}(3,2)\,\subset\, \text{GL}(5,{\mathbb R})$ be the real orthogonal group for
$q_{3,2}$, and denote by ${\rm SO}_0(3,2)$ the connected component
of ${\rm O}(3,2)$ containing the identity element. The quadric $Q$ admits a 
natural holomorphic action of the real Lie group ${\rm SO}_0(3,2)$, which is not transitive,
in contrast to the action of ${\rm SO}(5, \mathbb R)$ on $Q_3$. The orbits of the
${\rm SO}_0(3,2)$--action on $Q_3$ 
coincide with the connected components of the complement $Q_3 \setminus S$, where $S$ is
the real hypersurface of $Q_3$ defined by the equation
$$\mid Z_0\mid^2 + \mid Z_1\mid^2 + \mid Z_2\mid^2 - \mid Z_3\mid^3 - 
\mid Z_4\mid^2\,=\,0\, .$$
 
Notice that the above real hypersurface $S$ contains all real points of $Q$. In fact, it can be 
shown that $S \bigcap Q$ coincides with the set of point $m \,\in\, Q$ such that the complex 
line $(m,\, \overline{m})$ is isotropic (i.e., it actually lies in $Q$). Indeed, since 
$q_{3,2}(m)\,=\,0$, the line generated by $(m, \,\overline{m})$ lies in $Q$ if and only if 
$m$ and $\overline{m}$ are perpendicular with respect to the bilinear symmetric form associated 
to $q_{3,2}$, or equivalently $m \,\in\, S$.

For any point $m \,\in\, Q \setminus S$, the form $q_{3,2}$ is nondegenerate on the line $(m,\, 
\overline{m})$. To prove this first observe that the complex line generated by $(m,\, 
\overline{m})$, being real, may be considered as a plane in the real projective space ${\mathbb 
R}{\mathbb P}^4$. The restriction of the (real) quadratic form $q_{3,2}$ to this real plane 
$(m,\, \overline{m})$ vanishes at the points $m$ and $\overline{m}$ which are distinct (because 
all real points of $Q$ lie in $S$). It follows that the quadratic form cannot have signature 
$(0,\,1)$ or $(1,\,0)$ when restricted to the real plane $(m,\, \overline{m})$. Consequently, 
the signature of the restriction of $q_{3,2}$ to this plane is either $(2,\,0)$ or $(1,\,1)$ or 
$(0,\,2)$. Each of these three signature types corresponds to an ${\rm SO}_0(3,2)$ orbit in $Q$.
 
Take the point $m_0\,=\, [0:\,0:\,0:\,1:\,\sqrt{-1}] \,\in\, Q$. The noncompact dual $D_3$ of $Q_3$ is the 
${\rm SO_0}(3,2)$--orbit of $m_0$ in $Q$. It is an open subset of $Q$ biholomorphic to a 
bounded domain in ${\mathbb C}^3$; it is the three dimension Lie ball (the bounded domain 
$IV_3$ in Cartan's classification).

The signature of $q_{3,2}$ on the above line $(m_0,\,\overline{m}_0)$ is $(0,\,2)$. The 
signature of $q_{3,2}$ on the orthogonal part of $(m_0,\, \overline{m}_0)$, which is canonically 
isomorphic to $T_{m_0}Q$, is $(3,\,0)$. Then the ${\rm SO}_0(3,2)$--orbit of $m_0$ in $Q$ inherits 
an ${\rm SO}_0(3,2)$--invariant Riemannian metric. The stabilizer of $m_0$ is ${\rm SO}(2, 
\mathbb R) \times{\rm SO}(3, \mathbb R)$. Here ${\rm SO}(2, \mathbb R)$ acts on ${\mathbb 
C}^3$ through the one parameter group $\text{exp}(t J)$, with $J$ being the complex 
structure, while the ${\rm SO}(3, \mathbb R)$ action on ${\mathbb C}^3$ is given by the 
complexification of the canonical action of ${\rm SO}(3, \mathbb R)$ on ${\mathbb R}^3$. 
Consequently, the action of ${\rm U}(2)\,= \,{\rm SO}(2,{\mathbb R})\times{\rm SO}(3, 
{\mathbb R})$ is the natural irreducible action on the symmetric product $S^2({\mathbb 
C}^2)\,=\, {\mathbb C}^3$ constructed using the standard
representation of ${\rm U}(2)$. This action coincides with the holonomy representation of this 
Hermitian symmetric space $D_3$; as mentioned before, $D_3$ is the noncompact dual of $Q_3$. 
The holonomy representation for $Q_3$ is the same.

Recall that the automorphism group of the noncompact dual $D_3$ is ${\rm PSO}_0(3,2)$; it is 
the subgroup of the automorphism group of $Q$ that preserves $D_3$ (which lies in $Q$ as 
the ${\rm SO}_0(3,2)$--orbit of $m_0 \,\in\, Q$). Consequently, any quotient of $D_3$ by a 
lattice in ${\rm PSO}_0(3,2)$ admits a flat holomorphic conformal structure induced by that 
of the quadric $Q$. 

Note that the compact projective threefolds admitting a holomorphic conformal structure 
(a $\text{GL}(2,{\mathbb C})$-structure) were classified in \cite{JR1}. There are in fact 
only the standard examples: finite quotients of three dimensional abelian varieties, the 
smooth quadric $Q_3$ and quotients of its noncompact dual $D_3$. In \cite{JR2}, the same 
authors classified also the higher dimensional compact projective manifolds admitting a flat 
holomorphic conformal structure; they showed that the only examples are the standard ones.

\subsection{Classification results}

Theorem \ref{KE} shows that the only compact non-flat K\"ahler-Einstein manifolds bearing a holomorphic 
$\text{GL}(2,{\mathbb C})$-structure are $Q_3$ and those covered by its noncompact dual $D_3$.

\begin{theorem}[{\cite{BD1}}]\label{KE}
Let $M$ be a compact K\"ahler--Einstein manifold, of complex dimension at least three,
endowed with a holomorphic ${\rm GL}(2,{\mathbb 
C})$-structure. Then we are in one of the following (standard) situations:
\begin{enumerate}
\item $M$ admits a finite unramified covering by a compact complex torus and the pull-back of the ${\rm 
GL}(2,{\mathbb C})$-structure on the compact complex torus is the (translation invariant) standard one.

\item $M$ is biholomorphic to the three dimensional quadric $Q_3$ equipped with its standard ${\rm GL}(2,{\mathbb 
C})$--structure.
 
\item $M$ admits an unramified cover by the three-dimensional Lie ball $D_3$ (the noncompact dual of the Hermitian 
symmetric space $Q_3$) and the pull-back of the ${\rm GL}(2,{\mathbb C})$-structure on $D_3$ is the standard one.
\end{enumerate}
\end{theorem}

In order to explain the framework of Theorem \ref{Fano} and Theorem \ref{KE} let us recall that Kobayashi and Ochiai 
proved in \cite{KO3} a similar result on holomorphic conformal structures. More precisely, they showed that compact 
K\"ahler--Einstein manifolds bearing a holomorphic conformal structure are the standard ones: quotients of tori, the 
smooth $n$-dimensional quadric $Q_n$ and the quotients of the noncompact dual $D_n$ of $Q_n$.

Also Kobayashi and Ochiai proved in \cite{KO2} that all holomorphic $G$--structures, modeled on an irreducible 
Hermitian symmetric space of rank $\geq\, 2$ (in particular, a holomorphic conformal structure), on compact 
K\"ahler-Einstein manifolds are {\it flat}. The authors of \cite{HM} proved that all holomorphic irreducible 
reductive $G$--structures on uniruled projective manifolds are flat. Since uniruled projective manifolds are simply 
connected, this implies that a uniruled projective manifold bearing a holomorphic conformal structure is 
biholomorphic to the quadric $Q_n$ with its standard structure (see also \cite{Ye} and Proposition \ref{quadrics}). 
In \cite{BM}, the following generalization was proved:

{\it All holomorphic Cartan geometries (see 
\cite{Sh} or Section \ref{sect-Cartan}) on manifolds admitting a rational curve are flat.}

The methods used in \cite{BD1} to prove Theorem \ref{Fano} and Theorem \ref{KE} does not use the results in 
\cite{BM,HM, KO2,KO3, Ye}: they are specific to the case of $\text{GL}(2,{\mathbb C})$--geometry and unify the 
twisted holomorphic symplectic case (even dimensional case) and the holomorphic conformal case (odd dimensional 
case).

While every compact complex surface of course admits a holomorphic $\text{GL}(2,{\mathbb C})$-structure, the 
situation is much more stringent in higher dimensions. The following result proved in \cite{BD2} shows that a 
compact K\"ahler manifold of even dimension $n\,\geq\, 4$ bearing a holomorphic $\text{GL}(2,{\mathbb 
C})$--structure has trivial holomorphic tangent bundle (up to finite unramified cover).

\begin{theorem}\label{kahler even}
Let $M$ be a compact K\"ahler manifold of even complex dimension $n\,\geq\, 4$ admitting a 
holomorphic ${\rm GL}(2,{\mathbb C})$--structure. Then $M$ admits a finite unramified 
covering by a compact complex torus.
\end{theorem}

We will give below two different proofs of Theorem \ref{kahler even}.

\begin{proof}[{Proof of Theorem \ref{kahler even}}]
Let $E$ be a holomorphic vector bundle on $M$ such that 
$$TM\,\simeq\, S^{n-1} (E)$$ is 
an isomorphism with the symmetric product, 
defining the $\text{GL}(2,{\mathbb C})$--structure on $M$. Then
$$
TM \,=\, S^{n-1} (E)\,=\, S^{n-1} (E)^*\otimes (\bigwedge\nolimits^2E)^{\otimes (n-1)}
\,=\, (T^*M)\otimes L\, ,
$$
where $L\,=\, (\bigwedge\nolimits^2E)^{\otimes (n-1)}$. 

The above isomorphism between
$TM$ and $(T^*M)\otimes L$ produces, when $n$ is even, a holomorphic section
$$
\omega\, \in\, H^0(M,\, \Omega^2_M \otimes L)
$$
which is a fiberwise nondegenerate $2$--form with values in $L$. Writing $n\,=\, 2m$, the
exterior product
$$
\omega^m\,\in \, H^0(M,\, K_M \otimes L^m)
$$
is a nowhere vanishing section, where $K_M\,=\, \Omega^n_M$ is the canonical line bundle of $M$.

Consequently, we have $K_M\,\simeq\,(L^*)^m$, in particular, for the first real Chern class of $M$ we have $c_1(M)\,=\,mc_1(L)$.
Any Hermitian metric on $TM$ induces an associated Hermitian metric on $L^m$, and hence
produces an Hermitian metric on $L$.

We now use a result of Istrati, \cite[p.~747, Theorem 2.5]{Is1}, which says that
$c_1(M)\,=\,c_1(L)\,=\,0$.

Hence the manifold $M$ has vanishing first Chern class: it is a Calabi-Yau manifold. Recall that, 
Yau's proof of Calabi's conjecture, \cite{Ya}, endows $M$ with a Ricci flat K\"ahler metric $g$.

Using de Rham decomposition theorem and Berger's classification of the irreducible holonomy groups 
of nonsymmetric Riemannian manifolds (see \cite{Jo}, Section 3.2 and Theorem 3.4.1 in Section 
3.4) we deduce that the universal Riemannian cover $(\widetilde{M}, \,\widetilde{g}) $ of $(M,\,g)$ splits as 
a Riemannian product 
\begin{equation}\label{e1}
(\widetilde{M},\, \widetilde{g})\,=\, ({\mathbb C}^l,\, g_0) \times
(M_1,\, g_1) \times \cdots \times (M_p,\, g_p)\, ,
\end{equation}
where $({\mathbb C}^l, \, g_0)$ is the standard flat complete K\"ahler manifold and $(M_i,\, 
g_i)$ is an irreducible Ricci flat K\"ahler manifold of complex dimension $r_i\, \geq\, 2$, for every $1\, 
\leq\, i\, \leq\, p$. The holonomy of each $(M_i,\, g_i)$ is either $\text{SU}(r_i)$ 
or the symplectic group ${\rm Sp}(\frac{r_i}{2})$, where $r_i\, =\, \dim_{\mathbb C} M_i$
(in the second case $r_i$ is even). Notice, in 
particular, that symmetric irreducible Riemannian manifolds of (real) dimension at least two 
are never Ricci flat. For more details, the reader is referred to \cite[Theorem 1]{Be}
and \cite[p.~124, Proposition 6.2.3]{Jo}.

As a consequence of Cheeger--Gromoll theorem one can deduce the Beauville--Bogomolov decomposition theorem (see 
\cite[Theorem 1]{Be} or \cite{Bo1}) asserting that there is a finite unramified covering
$$
\varphi\, :\, \widehat{M}\, \longrightarrow\, M\, ,
$$
such that
\begin{equation}\label{e2}
(\widehat{M},\, \varphi^*g)\, =\, (T_l,\, g_0) \times
(M_1,\, g_1) \times \cdots \times (M_p,\, g_p)\, ,
\end{equation}
where $(M_i,\, g_i)$ are as in \eqref{e1} and $(T_l,\, g_0)$ is a flat compact complex torus of dimension
$l$. Of course the pull-back K\"ahler metric $\varphi^*g$ is Ricci--flat because $g$ is so.

We obtain that the initial holomorphic $\text{GL}(2,{\mathbb C})$--structure on $M$ induces a holomorphic 
$\text{SL}(2,{\mathbb C})$--structure on a finite unramified cover of $\widehat{M}$ in 
\eqref{e2}. Indeed, $T\widehat{M}\,=\, S^{n-1}(\varphi^* E)$ and since the canonical 
bundle $K_{\widehat{M}}$ is holomorphically trivial, the holomorphic line bundle $\bigwedge^2 E$ admits a finite multiple which is trivial (in this situation we also say that $\bigwedge^2 E$ is a {\it torsion} line bundle). 
Hence on a finite unramified cover of $\widehat{M}$ (still denoted by 
$\widehat{M}$ for simplicity) we can work with $\bigwedge^2 E$ being 
holomorphically trivial.

This $\text{SL}(2,{\mathbb C})$--structure on $\widehat{M}$ provides a holomorphic reduction $R'(\widehat{M})\, 
\subset\, R(\widehat{M})$ of the structure group of the frame bundle $R(\widehat{M})$ (from $\text{GL}(n,{\mathbb 
C})$) to $\text{SL}(2,{\mathbb C})$.

There is a finite set of holomorphic tensors $\theta_1, \ldots, \theta_s$ on $\widehat{M}$ satisfying the condition 
that the $\text{SL}(2,{\mathbb C})$--subbundle $R'(\widehat{M})\, \subset\,R(\widehat{M})$ consists of those frames 
that pointwise preserve all tensors $\theta_i$. This is deduced from Chevalley's theorem which asserts that there 
exists a finite dimensional linear representation $W$ of $\text{GL}(n,{\mathbb C})$, and an element $$\theta_0 
\,\in\, W\, ,$$ such that the stabilizer of the line ${\mathbb C} \theta_0$ is the image of the homomorphism 
$\text{SL}(2,{\mathbb C})\,\longrightarrow\, \text{GL}(n,{\mathbb C})$ defining the $\text{SL}(2, {\mathbb 
C})$-structure (see \cite[p.~80, Theorem 11.2]{Hu}, \cite[p.~40, Proposition 3.1(b)]{DMOS}; since 
$\text{SL}(2,{\mathbb C})$ does not have a nontrivial character, the line ${\mathbb C} \theta_0$ is fixed 
pointwise.

The group $\text{GL}(n,{\mathbb C})$ being reductive, we decompose $W$ as a direct sum $\bigoplus_{i=1}^s W_i$ of 
irreducible representations. Now, since any irreducible representation $W_i$ of the reductive group 
$\text{GL}(n,{\mathbb C})$ is a factor of a representation $({\mathbb C}^n)^{\otimes p_i} \otimes (({\mathbb 
C}^n)^*)^{\otimes q_i},$ for some integers $p_i,\,q_i\, \geq \, 0$ \cite[p.~40, Proposition 3.1(a)]{DMOS}, the 
above element $\theta_0$ gives rise to a finite set $\theta_1,\, \ldots ,\, \theta_s$ of holomorphic tensors
\begin{equation}\label{th}
\theta_i \,\in\, H^0(\widehat{M},\,(T{\widehat{M}})^{\otimes p_i}\otimes (T^*{\widehat{M}})^{\otimes q_i})
\end{equation}
with $p_i,q_i\, \geq \, 0$. By construction, $\theta_1, \ldots, \theta_s$ are simultaneously stabilized exactly by the frames lying in $R'(\widehat{M})$.

It is known that the parallel transport on $\widehat{M}$ for the Levi--Civita connection associated to the Ricci--flat K\"ahler 
metric $\varphi^*g$ in \eqref{e2}. preserves any 
holomorphic tensor on $\widehat{M}$ \cite[p.~50, Theorem 2.2.1]{LT}. In particular, $\theta_i$ in 
\eqref{th} are all parallel with respect to the Levi-Civita connection of $\varphi^*g$. We conclude that the subbundle 
$R'(\widehat{M})\, \subset\, R(\widehat{M})$ defining the holomorphic $\text{SL}(2,{\mathbb C})$-structure 
 is invariant under the parallel transport of the Levi--Civita 
connection for $\varphi^*g$. This implies that the holonomy group of $\varphi^*g$ lies in the 
maximal compact subgroup of $\text{SL}(2,{\mathbb C})$. Consequently, the holonomy group of $\varphi^*g$ 
lies in $\text{SU}(2)$.

From \eqref{e2} it follows that the holonomy of $\varphi^*g$ is
\begin{equation}\label{e3}
\text{Hol}(\varphi^*g)\,=\, \prod_{i=1}^p \text{Hol}(g_i)\, ,
\end{equation}
where $\text{Hol}(g_i)$ is the holonomy of $g_i$. As noted earlier,
\begin{itemize}
\item either $\text{Hol}(g_i)\,=\, \text{SU}(r_i)$, with $\dim_{\mathbb C} M_i \,=\, r_i\, \geq\, 2$, or

\item $\text{Hol}(g_i)\,=\, {\rm Sp}(\frac{r_i}{2})$, where $r_i\, =\, \dim_{\mathbb C} M_i$ is
even.
\end{itemize}
Therefore, the above observation, that $\text{Hol}(\varphi^*g)$ is contained in $\text{SU}(2)$,
and \eqref{e3} together imply that
\begin{enumerate}
\item either $(\widehat{M},\, \varphi^*g)\, =\, (T_l,\, g_0)$, or

\item $(\widehat{M},\, \varphi^*g)\, =\, (T_l,\, g_0) \times (M_1,\, g_1)$, where $M_1$ is a K3
surface equipped with a Ricci--flat K\"ahler metric $g_1$.
\end{enumerate}

If $(\widehat{M},\, \varphi^*g)\, =\, (T_l,\, g_0)$, then then proof of the theorem evidently is complete.

Therefore, we assume that
\begin{equation}\label{a1}
(\widehat{M},\, \varphi^*g)\, =\, (T_l,\, g_0) \times (M_1,\, g_1)\, ,
\end{equation}
where $M_1$ is a K3 surface equipped with a Ricci--flat K\"ahler metric $g_1$. Note that
$l\, \geq\, 2$ (because $l+2\,=\, n\, \geq\, 4$) and $l$ is even (because $n$ is so).

At this stage there are two different ways to terminate the proof. We will give below the two of them.

I){\it First Proof.}

Since $\text{Hol}(g_1)\,=\, \text{SU}(2)$, we get from (\ref{e3}) that $\text{Hol}(\varphi^*g)\,=\, \text{SU}(2)$. The holonomy of $\varphi^*g$ is 
the image of the homomorphism
\begin{equation}\label{h0}
h_0\, :\, \text{SU}(2)\, \longrightarrow\, \text{SU}(n)
\end{equation}
given by the $(n-1)$--th symmetric power of the standard representation. The action of
$h_0(\text{SU}(2))$ on ${\mathbb C}^n$, obtained by restricting the standard action of $\text{SU}(n)$,
is irreducible. In particular, there are no nonzero
$\text{SU}(2)$--invariants in ${\mathbb C}^n$.

On the other hand we have that:
\begin{itemize}
\item the direct summand of $T \widehat{M}$ given by the tangent bundle $TT_l$ is
preserved by the Levi--Civita connection on $T \widehat{M} $ corresponding to $\varphi^*g$, and

\item this direct summand of $T \widehat{M}$ given by $TT_l$ is generated by flat sections of $TT^l$.
\end{itemize}

Since $T \widehat{M}$ does not have any flat section, we conclude that $l=0$: a contradiction. This terminates the first proof.

II) {\it Second proof}.

Fix a point $t\, \in\, T_l$. The holomorphic vector bundle over $M_1$ obtained by restricting $\varphi^* E$
to $\{t\}\times M_1\, \subset\, T_l\times M_1$ will be denoted by $F$. Restricting the isomorphism
$$
T\widehat{M}\,\stackrel{\sim}{\longrightarrow}\, S^{l+1}(\varphi^* E)
$$
to $\{t\}\times M_1\, \subset\, T_l\times M_1$, from \eqref{a1} we conclude that
\begin{equation}\label{a2}
S^{l+1}(F)\,=\, T{M_1}\oplus {\mathcal O}^{\oplus l}_{M_1}\, ;
\end{equation}
note that $(T{\widehat{M}})\vert_{\{t\}\times M_1}\,=\, TM_1\oplus {\mathcal O}^{\oplus l}_{M_1}$.
The vector bundle $TM_1$ is polystable of degree zero (we may use the K\"ahler structure $g_1$ on
$M_1$ to define the degree), because $M_1$ admits a K\"ahler--Einstein metric. Hence
$T{M_1}\oplus {\mathcal O}^{\oplus(l+1)}_{M_1}$ is polystable of degree zero. This implies that
$\text{End}(T{M_1}\oplus {\mathcal O}^{\oplus(l+1)}_{M_1})$ is polystable of degree zero, because
for any Hermitian--Einstein structure on $T{M_1}\oplus {\mathcal O}^{\oplus(l+1)}_{M_1}$ (which exists by
\cite{UY}), the Hermitian structure on $\text{End}(T{M_1}\oplus {\mathcal O}^{\oplus(l+1)}_{M_1})$ 
induced by it is also Hermitian--Einstein. On the other hand, the holomorphic vector bundle $\text{End}(F)$ 
is a direct summand of $\text{End}(S^{l+1}(F))$. Since $\text{End}(S^{l+1}(F))\,=\,
\text{End}(T{M_1}\oplus {\mathcal O}^{\oplus(l+1)}_{M_1})$ is polystable of degree zero, this
implies that the holomorphic vector bundle $\text{End}(F)$ is also polystable of degree zero.

Since $\text{End}(F)$ is polystable, it follows that $F$ is polystable \cite[p.~224, Corollary 3.8]{AB}.

Let $H$ be a Hermitian--Einstein metric on $F$. The holonomy of the Chern connection on $F$ associated to $H$
is contained in $\text{SU}(2)$, because $\bigwedge^2 F\, =\, {\mathcal O}_{M_1}$ (since $\bigwedge^2 E$ is trivial). Let $\text{Hol}(H)$
denoted the holonomy of the Chern connection on $F$ associated to $H$. We note that $H$ is connected because
$M_1$ is simply connected. Therefore, either $\text{Hol}(H)\,=\, \text{SU}(2)$ or
$\text{Hol}(H)\,=\, \text{U}(1)$ (since $TM_1$ is not trivial, from
\eqref{a2} it follows that $F$ is not trivial and hence the holonomy of $H$ cannot be trivial).

First assume that $\text{Hol}(H)\,=\, \text{SU}(2)$. Let $\mathcal H$ denote the Hermitian--Einstein
metric on $S^{n-1}(F)$ induced by $H$. Since $\text{Hol}(H)\,=\, \text{SU}(2)$, the holonomy
$\text{Hol}({\mathcal H})$ of the Chern connection on $S^{n-1}(F)$ associated to $\mathcal H$ is
the image of the homomorphism
\begin{equation}\label{h01}
h_0\, :\, \text{SU}(2)\, \longrightarrow\, \text{SU}(n)
\end{equation}
given by the $(n-1)$--th symmetric power of the standard representation. The action of
$h_0(\text{SU}(2))$ on ${\mathbb C}^n$, obtained by restricting the standard action of $\text{SU}(n)$,
is irreducible. From this it follows that
$$
H^0(M_1,\, S^{n-1}(F))\, =\, 0
$$
because any holomorphic section of $S^{n-1}(F)$ is flat with respect to the
Chern connection on $S^{n-1}(F)$ associated to $\mathcal H$ \cite[p.~50, Theorem 2.2.1]{LT}.
On the other hand, from \eqref{a2} we have
\begin{equation}\label{a4}
H^0(M_1,\, S^{n-1}(F))\, =\, {\mathbb C}^l\, .
\end{equation}
In view of this contradiction we conclude that $\text{Hol}(H)\,\not=\, \text{SU}(2)$.

So assume that $\text{Hol}(H)\,=\, \text{U}(1)$. Then
\begin{equation}\label{a3}
F\,=\, {\mathcal L}\oplus {\mathcal L}^*\, ,
\end{equation}
where ${\mathcal L}$ is a holomorphic line bundle on $M_1$. Note that $\text{degree}({\mathcal L})\,=\, 0$,
because $F$ is polystable. From \eqref{a3} it follows that
$$
S^{n-1}(F)\,=\, \bigoplus_{j=0}^{n-1} {\mathcal L}^{1-n+2j}\, .
$$
Now using \eqref{a4} we conclude that $H^0(M_1, \, {\mathcal L}^k)\, \not=\, 0$ for some nonzero integer $k$.
Since $\text{degree}({\mathcal L})\,=\, 0$, this implies that the holomorphic line bundle $\mathcal L$ is trivial;
note that the Picard group $\text{Pic}(M_1)$ is torsionfree. Hence $S^{n-1}(F)$ is trivial. Since
$T{M_1}$ is not trivial, this contradicts \eqref{a2} (see \cite[p.~315, Theorem 2]{At0}).

Therefore, \eqref{a1} can't occur. This completes the second proof of the theorem.
\end{proof} 

Recall that a compact K\"ahler manifold of odd complex dimension bearing a holomorphic $\text{SL}(2,{\mathbb 
C})$--structure also admits a holomorphic Riemannian metric and inherits of the associated holomorphic (Levi-Civita) 
 connection on the holomorphic tangent bundle. Those manifolds are known to have vanishing Chern classes 
\cite{At}. Indeed, following Chern-Weil theory one classically computes real Chern classes of the holomorphic 
tangent bundle of a manifold $M$ using a hermitian metric on it. This provides representatives of the Chern class 
$c_k(TM)\,\in\, \text{H}^{2k}(M,\, \mathbb R)$ which are smooth forms on $M$ of type $(k,\,k)$. Starting with a 
holomorphic connection on $TM$ and doing the same formal computations, one gets another representative of $c_k(TM)$ 
which is a holomorphic form: in particular, it is of type $(2k,\,0)$. Classical Hodge theory says that on K\"ahler 
manifolds nontrivial real cohomology classes do not have representatives of different types. This implies the 
vanishing of the real Chern class: $c_k(TM)\,=\,0$ for all $k$.

But a compact K\"ahler manifold $M$ with vanishing first two Chern classes ($c_1(TM)\,=\,c_2(TM)\,=\,0$) is known to be 
covered by a compact complex torus \cite{IKO}. Indeed, using the vanishing of the first Chern class of $M$, Yau's 
proof of Calabi conjecture (which is the key ingredient to obtain the result) \cite{Ya} endows $M$ with a Ricci flat 
K\"ahler metric. The vanishing of the second Chern class implies that the Ricci flat metric is flat (i.e., it has 
vanishing sectional curvature). Hence $M$ admits a flat K\"ahler metric. Since $M$ is compact, Hopf--Rinow theorem 
implies the flat metric is complete, meaning the universal cover of $M$ is isometric with $\CC^n$ endowed with its 
standard translation invariant K\"ahler metric. By Bieberbach's theorem $M$ admits a finite cover which is a 
quotient of $\CC^n$ by a lattice of translations.

Therefore, Theorem \ref{kahler even} has the following corollary.

\begin{corollary}\label{sl2}
Let $M$ be a compact K\"ahler manifold of complex dimension $n \,\geq\, 3$
bearing a holomorphic ${\rm SL}(2,{\mathbb C})$--structure. Then $M$ admits a finite unramified
covering by a compact complex torus.
\end{corollary} 

\section{Holomorphic Cartan geometry}\label{sect-Cartan}

In his celebrated Erlangen's address (delivered in  1872), F. Klein defined a {\it geometry} as a manifold $X$ equipped with a 
transitive action of a Lie group $G$. The Lie group $G$ is the symmetry group of the geometry.
Notice that any choice of a base point $p_0 \,\in\, X$ identifies the manifold $X$ with the homogeneous space $G/H$, 
where $H$ is the closed subgroup of $G$ stabilizing $p_0$.

The Euclidean geometry is a first important example of geometry in the sense of Klein. The Euclidean symmetry group 
is the group of all motions. The stabilizer of the origin is the orthogonal group $H \,=\, \text{O}(n, {\mathbb R})$ 
and the group of Euclidean motions $G \,=\, \text{O}(n, \mathbb{R})\ltimes {\mathbb R}^n$ is a semi-direct product 
of the orthogonal group with the translation group.

The holomorphic version of the above geometry is the {\it complex Euclidean space} $(X ={\mathbb C}^n,\, dz_1^2 + 
\ldots + dz_n^2)$. Here $dz_1^2+ \ldots + dz_n^2$ is the nondegenerate complex quadratic form (of maximal rank $n$) 
and $H\,=\,\text{O}(n , \mathbb{C})$ is the corresponding complex orthogonal group. The symmetry group 
$G\,=\,\text{O}(n, \mathbb{C}) \ltimes {\mathbb C}^n$ is the group of complex Euclidean motions. With the point of 
view of $G$-structure, the complex quadratic form $dz_1^2 + \ldots + dz_n^2$ defines a flat holomorphic ${\rm O}(n, 
\CC)$-structure in the sense of Section \ref{section G-struct} and hence a flat holomorphic Riemannian metric (see 
\cite{Le,Gh,Du,DZ,BD2}). Any flat holomorphic Riemannian metric is locally isomorphic with the complex Euclidean 
space.

Another classical geometry encompassed in Klein's definition is the {\it complex affine space} $X\,=\,{\mathbb C}^n$. 
The symmetry group is the complex affine group $G\,=\,\text{GL}(n,\mathbb C) \ltimes {\mathbb C}^n$. This symmetry group 
preserves lines parametrized at constant speed.
 
Klein's definition contains also the {\it complex projective space}. The symmetry group of the complex projective 
space ${\mathbb C}{\rm P}^n$ is the complex projective group $\text{PGL}(n+1, \mathbb C)$. This symmetry group preserves 
lines.
 
As a unifying generalization of the above concept of Klein geometry and of the Riemannian geometry, E. Cartan 
elaborated the framework of what we now call {\it Cartan geometry}. A Cartan geometry is an infinitesimal version of 
a Klein geometry: this generalizes Riemann's construction of a Riemannian metric infinitesimally modeled on the 
Euclidean space. We will see that any Cartan geometry has an associated curvature tensor measuring the 
infinitesimal variation of the Cartan geometry with respect to the corresponding Klein's model. The curvature tensor 
vanishes exactly when the Cartan geometry is locally isomorphic to a Klein geometry. This generalizes to the Cartan 
geometry background the Riemannian notion of curvature: recall that the Riemannian curvature vanishes exactly when 
the metric is flat and hence locally isomorphic with the Euclidean space.
 
The Cartan geometries for which the curvature tensor vanishes identically are called {\it flat}.

\subsection{Classical case} Let us now introduce the classical definition of a Cartan geometry
in the complex analytic category (the reader is 
referred to \cite{Sh} for more details).

The model of the Cartan geometry is a Klein geometry $G/H$, where $G$ is a connected complex Lie group and $H\, 
\subset\, G$ is a connected complex Lie subgroup. The complex Lie algebras of $G$ and $H$ are denoted by 
$\mathfrak g$ and $\mathfrak h$ respectively.

Then we have the following definition due to Cartan and formalized by Ehresmann.

\begin{definition} \label{def Cartan}
A holomorphic Cartan geometry with model $(G,\,H)$ on a complex manifold $M$ is a holomorphic principal $H$--bundle 
$\pi \,: E_H\, \longrightarrow\, M$ equipped with a $\mathfrak g$--valued holomorphic $1$--form $\omega \in 
\text{H}^0(E_H,\, \mathfrak g)$ satisfying:
\begin{enumerate}
\item $\omega \, :\, TE_H\, \longrightarrow\, E_H\times{\mathfrak g}$ defines a vector bundle
isomorphism;

\item $\omega$ is $H$--equivariant with $H$ acting on $\mathfrak g$ via conjugation;

\item the restriction of $\omega$ to each fiber of $\pi$ coincides with the Maurer--Cartan form
associated to the action of $H$ on $E_H$.
\end{enumerate}
\end{definition}

A $\omega$-constant vector field on $E_H$ is a section of $TE_H$ which is the preimage through $\omega$ of a fixed 
element in $\mathfrak g$. Condition (1) in Definition \ref{def Cartan} means that the holomorphic tangent bundle 
$TE_H$ is trivialized by $\omega$-constant vector fields. This condition implies, in particular, that the complex 
dimension of the model space $G/H$ is the same as the complex dimension of $M$.

Condition (3) in Definition \ref{def Cartan} implies that the $\omega$-constant vector fields corresponding to 
preimages of elements in $\mathfrak h\,\subset\, \mathfrak g$ form a family of holomorphic vector fields on $E_H$  trivializing 
the holomorphic vertical tangent space (the kernel of $d\pi$ in $TE_H$). The restriction of $\omega$ to
the vertical tangent bundle identifies the vertical tangent bundle with the trivial
bundle $E_H \times \mathfrak h\, \longrightarrow\, E_H$.

Moreover, condition (2) in Definition \ref{def Cartan} implies that if $Z_1,\, Z_2\, \in\, \text{H}^0(E_H,\, TE_H)$ 
are two $\omega$-constant vector fields such that one of them, say $Z_1$, is vertical, then $\omega(\lbrack Z_1,\, 
Z_2])\,=\,\lbrack \omega(Z_1),\, \omega(Z_2) \rbrack_{\mathfrak{g}}$.

In fact, $\omega$ is a Lie algebra isomorphism from the family of $\omega$-constant vector fields to $\mathfrak g$ 
if and only if the holomorphic $2$-form $d \omega + \frac{1}{2} \lbrack \omega,\, \omega \rbrack_{\mathfrak g}$ 
vanishes identically. In this case the Cartan geometry defined by $\omega$ is called {\it flat}. The holomorphic 
$2$-form $$K(\omega)\,:=\, d \omega + \frac{1}{2} \lbrack \omega,\, \omega \rbrack_{\mathfrak g} \,\in\, 
\text{H}^0(E_H,\, \Omega^2_{E_H})$$ is called the curvature tensor of the Cartan geometry. Note that the above 
observation imply that $K(\omega)$ vanishes on all pairs $(Z_1,\,Z_2)$, where $Z_1$ is a $\omega$-constant vertical 
holomorphic vector field and $Z_2$ is any $\omega$-constant  holomorphic  vector field. Consequently, $K(\omega)$ is the 
pull-back, through $\pi$, of an element of $\text{H}^0(M, \,\Omega^2(M) \otimes \text{ad}(E_G))$; here $\text{ad}(E_G)$ is the 
vector bundle associated to $E_H$ through the action of $H \,\subset\, G$ on $\mathfrak g$
via the adjoint representation. In 
particular, any holomorphic Cartan geometry on a complex manifold of dimension one (complex curve) is necessarily 
flat.

Definition \ref{def Cartan} should be seen as an infinitesimal version of the principal $H$--bundle given by the 
quotient map $G \,\longrightarrow\, G/H$. The form $\omega$ generalizes the left-invariant Maurer-Cartan form 
$\omega_G$ of $G$. It should be recalled here that the Maurer-Cartan form $\omega_G$ is the tautological one-form on 
$G$ which identifies left-invariant vector fields on $G$ with elements in the Lie algebra $\mathfrak g$. It 
satisfies the so-called Maurer-Cartan equation $d \omega + \frac{1}{2} \lbrack \omega,\, \omega \rbrack_{\mathfrak 
g}\,=\,0$ (this is a straightforward consequence of the Lie-Cartan derivative formula, see, for instance, \cite{Sh}).

We recall a classical result of Cartan (see \cite{Sh}):

\begin{theorem}[{Cartan}]Let $(E_H,\,\omega)$ be a Cartan geometry with model $(G,\,H)$. The curvature $K(\omega)\,=\,
d\omega + \frac{1}{2} \lbrack \omega,\,\omega \rbrack_{\mathfrak g}$ vanishes identically if and only if $(E_H, \omega 
)$ is locally isomorphic to $(G \longrightarrow G/H,\, \omega_G)$, where $\omega_G$ is the left-invariant Maurer-Cartan form on 
the (right) principal $H$--bundle $G \,\longrightarrow\, G/H$.
\end{theorem}

In \cite{Eh} Ehresmann studied manifolds $M$ equipped with a flat Cartan geometry. More precisely, a flat Cartan 
geometry on $M$, gives a $(G,\,X=G/H)$-structure on $M$ in the following sense described by Ehresmann (in \cite{Eh} 
the terminology is that of a {\it locally homogeneous space}).

\begin{definition}\label{def G,X}
A holomorphic $(G,\,X)$-structure on a complex manifold $M$ is given by an open cover $(U_i)_{i \in I}$ of $M$ with 
holomorphic charts $\phi_i \,:\, U_i \,\longrightarrow\, X$ such that the transition maps $\phi_i \circ \phi_j^{-1}
\,:\, \phi_j(U_i \cap U_j)\,\longrightarrow\, \phi_i(U_i \cap U_j)$ are given (on each connected component) by
the restriction of an element $g_{ij} \,\in\, G$.
\end{definition}
 
Recall that the $G$-action on $X$ is holomorphic. Also note that any geometric feature of $X$ which is invariant 
by the symmetry group $G$ has an intrinsic meaning on the manifold $M$ equipped with a $(G,\,X)$-structure.

Consider a $(G,\,X)$-structure on a manifold $M$ (as in Definition \ref{def G,X}). Then the pull-back of the 
$(G,\,X)$-structure to the universal cover $\widetilde{M}$ of $M$ gives rise to a local biholomorphism $\text{dev}\,: 
\,\widetilde{M} \,\longrightarrow\, X$ which is called the {\it developing map} (of the $(G,\,X)$-structure) and to a 
group homomorphism form the fundamental group of $M$ to $G$ called the {\it monodromy morphism}. The developing map 
is uniquely defined up to a post-composition by an element in $G$ (acting on $X$) and the monodromy morphism is 
well-defined up to an inner automorphism of $G$. Moreover, the developing map is equivariant with respect to the action of 
the fundamental group of $M$ on $\widetilde{M}$ (by deck transformations) and on $X$ (through the monodromy 
morphism). A more detailed construction of the developing map of a flat Cartan geometry will be given in Section 
\ref{section foliated Atyiah} focusing on the foliated case.

By Cartan's equivalence principle, many important geometric structures can be seen as $G$-structures and also as 
Cartan geometries. For example, a holomorphic Riemannian metric on a complex manifold of complex dimension $n$ is a 
holomorphic $\text{O}(n, \mathbb{C})$-structure on $M$ and also a holomorphic Cartan geometry with model 
$(G\,=\,\text{O}(n, \mathbb{C}) \ltimes {\mathbb C}^n,\, H\,=\,\text{O}(n, \mathbb{C})).$ Those two points of view 
are equivalent; in particular the notion of flatness is the same in these two descriptions of the holomorphic 
Riemannian metric.

Also a holomorphic Cartan geometry whose model is the complex affine space (meaning $G\,=\,\text{GL}(n,\mathbb C) 
\ltimes {\mathbb C}^n$ and $H\,=\,\text{GL}(n, {\mathbb C})$) is actually a holomorphic affine connection on $M$. 
This Cartan geometry is flat if and only if the holomorphic affine connection is flat and torsion-free. In this case 
$M$ is endowed with a holomorphic affine structure (i.e., a holomorphic $(G,\,X)$--structure, with 
$X\,=\,G/H\,=\,\CC^n$ being the complex affine space and $G\,=\,\text{GL}(n,\mathbb C) \ltimes {\mathbb C}^n$ its 
symmetry affine group).

A holomorphic Cartan geometry whose model is the complex projective space (meaning $G\,=\,\text{PGL}(n+1,\mathbb C)$ 
and $H$ is its parabolic subgroup stabilizing a line in $\CC{\rm P}^n$) defines a holomorphic projective connection on 
$M$. When the connection flat, this Cartan geometry defines a holomorphic projective structure on $M$ (i.e, a 
holomorphic $(G,\,X)$-structure, with $X\,=\,G/H\,=\,\CC{\rm P}^n$ being the complex projective space while 
$G\,=\,\text{PGL}(n+1,\mathbb C)$ is its automorphism group).

\subsection{Branched Cartan geometry}

In \cite{M1, M2} Mandelbaum defined and studied complex {\it branched affine and projective structures} on Riemann 
surfaces. A holomorphic branched affine (respectively, projective) structure on a Riemann surface is given by some 
open cover (as in Definition \ref{def G,X}) and local charts $\phi_i\,:\, U_i\,\longrightarrow\, \CC$ which are finite branched 
coverings such that the transition maps $\phi_i \circ \phi_j^{-1}\,:\, \phi_j(U_i \cap U_j)\,\longrightarrow\, \phi_i(U_i \cap U_j)$ 
are given by restrictions of an element $g_{ij}$ lying in the complex affine group of $\CC$ (respectively, in 
$\text{PSL}(2, {\mathbb C})$).

Generalizing Mandelbaum's definition, we introduced in \cite{BD} the notion of {\it branched holomorphic Cartan 
geometry} on a complex manifold $M$. This notion is well-defined for manifolds $M$ of any complex dimension and not 
only in the flat case.

The precise definition is the following \cite{BD}:

\begin{definition}\label{def Cartan branched}
A branched holomorphic Cartan geometry with model $(G,H)$ on a complex manifold $M$ is a holomorphic principal 
$H$--bundle $\pi \,: E_H\, \longrightarrow\, M$ equipped with a $\mathfrak g$--valued holomorphic $1$--form 
$\omega\,\in\, \text{H}^0(E_H,\, \mathfrak g)$ satisfying:
\begin{enumerate}
\item $\omega \, :\, TE_H\, \longrightarrow\, E_H\times{\mathfrak g}$ defines a vector bundle
morphism which is an isomorphism on an open dense set in $E_H$;

\item $\omega$ is $H$--equivariant with $H$ acting on $\mathfrak g$ via conjugation.

\item the restriction of $\omega$ to each fiber of $\pi$ coincides with the Maurer--Cartan form
associated to the action of $H$ on $E_H$.
\end{enumerate}
\end{definition}

Condition (1) implies that the complex dimension of the model $G/H$ is the same as the complex dimension of $M$.

Condition (2) implies that the open dense set $U$ in $E_H$ over which $\omega$ is an isomorphism is $H$-invariant. 
More precisely, there exists a divisor $D\,\subset\,M$ such that $U\,=\,\pi^{-1} (M \setminus D)$.

The divisor $D$ is called {\it the branching divisor} of the branched Cartan geometry $(E_H, \omega)$.

The branched Cartan geometry $(E_H,\, \omega)$ is called {\it flat} if its curvature tensor, defined again as 
$K(\omega)\,=\,d \omega + \frac{1}{2} \lbrack \omega,\, \omega \rbrack_{\mathfrak g}$, vanishes identically. In this 
case it was proved in \cite{BD} that there exists a holomorphic developing map $\text{dev}\,:\, 
\widetilde{M}\,\longrightarrow\, X$ and a group homomorphism form the fundamental group of $M$ to $G$ (the 
monodromy morphism of the branched flat Cartan geometry) . The developing map is uniquely defined up to a 
post-composition by an element in $G$ (acting on X) and the monodromy morphism is well-defined up to inner 
conjugacy in $G$. Moreover, the developing map is equivariant with respect to the action of the fundamental group 
of $M$ on $\widetilde{M}$. In the branched case the developing map is a holomorphic dominant map; its differential 
is invertible exactly over the open dense set $M \setminus D$ (away from the branching divisor $D$).

In particular, $M$ is endowed with a branched holomorphic $(G,\,X)$-structure in the sense of the following 
definition:

\begin{definition}\label{def branched G,X}
A branched holomorphic $(G,X)$-structure on a complex manifold $M$ is 
given by an open cover $(U_i)_{i \in I}$ of $M$ with branched holomorphic maps $\phi_i \,:\, U_i\,\longrightarrow\, X$ such that the 
transition maps $\phi_i \circ \phi_j^{-1} \,:\, \phi_j(U_i \cap U_j)\,\longrightarrow\,\phi_i(U_i \cap U_j)$ are given (on each 
connected component) by the restriction of an element $g_{ij}\,\in\, G$.
\end{definition}

Notice that the element $g_{ij}$ in the above definition is unique (since any element of $G$ is uniquely determined 
by the restriction of its action on a nontrivial open set in $X$).

In the most basic example, where $G$ is the complex Lie group $\CC$ and $H= \{0 \}$, the associated Cartan geometry 
is simply a nontrivial holomorphic $1$-form $\omega$ on a Riemann surface. The branching divisor $D$ is the divisor 
of zeros of $\omega$. The developing map is a primitive of $\omega$ and the image of the monodromy morphism is the 
group of periods of $\omega$.

A branched holomorphic Cartan geometry with model the complex projective space (recall that 
$G\,=\,\text{PGL}(n+1,\mathbb C)$ and $H$ is its parabolic subgroup stabilizing a line in $\text{P}^n(\CC)$) defines a 
branched holomorphic projective connection on $M$. When flat this Cartan geometry defines a branched holomorphic 
projective structure on $M$ (i.e., a branched holomorphic $(G,X)$-structure, with $X=G/H=\text{P}^n(\CC)$ being the 
complex projective space and $G\,=\,\text{PGL}(n+1,\mathbb C)$ its symmetry projective group).

This notions of branched Cartan geometry is much more flexible than the standard one. It is stable by pull-back 
through holomorphic ramified maps This was used in \cite{BD} to prove that all compact complex projective manifolds 
admit a branched flat holomorphic projective structure.

A more general version of Cartan geometry was defined in \cite{AM,BD4} where this concept is referred as a {\it 
generalized Cartan geometry}. A holomorphic generalized Cartan geometry $(E_H,\,\omega)$ with model $(G,H)$ on a 
complex manifold $M$ satisfies conditions (2) and (3) in Definitions \ref{def Cartan} and \ref{def Cartan 
branched}, but conditions (1) are dropped. In particular, there is no relation between the complex dimension of $M$ 
and the complex dimension of the model space $G/H$. In the flat case (defined by the vanishing of the curvature 
tensor $K(\omega)\,=\,d \omega + \frac{1}{2} \lbrack \omega, \omega \rbrack_{\mathfrak g}$), the developing map (of a 
holomorphic generalized Cartan geometry) is still a holomorphic map $\widetilde{M} \,\longrightarrow\, X\,=\,G/H$ 
which is equivariant with respect to the monodromy morphism; but here there is no condition on the rank of the 
developing map (see \cite{BD4}).

Consider again the basic example, where $G$ is the complex Lie group $\CC$ and $H\,=\, \{0 \}$. The associated Cartan 
geometry on a complex manifold $M$ is a holomorphic $1$-form $\omega\,\in\, \text{H}^0(\Omega^1(M,\, \CC))$. The 
branching divisor $D$ is the divisor of zeros of $\omega$. The Cartan geometry is flat if and only if $d 
\omega\,=\,0$. In particular, all those geometries are flat on compact K\"ahler manifolds and on compact complex 
surfaces. In this flat case, the developing map is a primitive of $\omega$ and the image of the monodromy morphism 
is the group of periods of $\omega$.

\section{Transverse Cartan geometry}\label{section foliated diff bundle}

The definition of a transverse Cartan geometry was first given by Blumenthal in \cite{Bl} in the context of foliated 
differential bundles (see, for instance, \cite{Mol}).
In \cite{BD3} we worked out a definition of a transverse branched holomorphic Cartan geometry using the formalism of 
Atiyah's bundle.
We present here this notion of transverse branched Cartan geometry in the complex analytic category.

Let $M$ be a complex manifold and $\mathcal F$ a holomorphic foliation on $M$ of complex codimension $n$.
For simplicity we assume first that $\mathcal F$ does not admit singularities. We will explain later how holomorphic 
foliations with singularities can be endowed with transverse Cartan geometry.

Recall that the associated tangent space $T\mathcal F \,\subset\, TM$ is a holomorphic distribution of codimension $n$ 
which is stable by Lie bracket. Conversely, any holomorphic distribution stable by Lie bracket is {\it integrable}, 
meaning it coincides with the tangent space of a foliation $\mathcal F$.
 
The normal bundle of the foliation, defined as the quotient $\mathcal N_{\mathcal F}
\,=\,TM/T \mathcal F$, admits a canonical holomorphic flat connection $\nabla^{\mathcal F}$ along the leafs
of the foliation $\mathcal F$. To define analytically this connection
one chose two local sections $X$ and $N$ of $T\mathcal F$ and $\mathcal N_{\mathcal F}$ respectively and
defines the derivative of $N$ in the direction of $X$ as being
$$\nabla^{\mathcal F}_X N\,=\, q(\lbrack X,\, \widetilde{N} \rbrack),$$ 
where $q \,:\, TM \,\longrightarrow\, \mathcal N_{\mathcal F}\,=\,TM/T \mathcal F$ is the quotient map
and $\widetilde{N}$ is any local section of $TM$ such that $q(\widetilde{N})\,=\,N$. 
 
The above connection is well-defined along $\mathcal F$. Indeed, if $\widehat{N}$ is another choice of lift for $N$ 
to $TM$, then $ q(\lbrack X, \,\widehat{N} \rbrack)\,=\, q(\lbrack X, \,\widetilde{N} \rbrack)$, since 
$\widehat{N}-\widetilde{N} \in T \mathcal F$ the Lie bracket stability of $T\mathcal F$ implies that $\lbrack 
\widehat{N}-\widetilde{N}, X \rbrack \in T\mathcal F$. Consequently, $q(\lbrack \widehat{N}-\widetilde{N}, X 
\rbrack)=0$ which implies $ q(\lbrack X,\, \widehat{N} \rbrack)\,=\,q(\lbrack X,\, \widetilde{N} \rbrack)$.
 
Notice that the local section $N\,=\,q(\widetilde N)$ of $\mathcal N_{\mathcal F}$ is parallel with respect to 
$\nabla^{\mathcal F}$ if and only if $\lbrack \widetilde{N} ,\, X \rbrack\,\in\, \mathcal F$, for all $X\,
\in\, \mathcal F$.
 
Moreover, Jacobi's identity for the Lie bracket implies the vanishing of the curvature tensor of $\nabla^{\mathcal 
F}$ and hence its flatness.
 
A transverse holomorphic Cartan geometry on $(M,\, \mathcal F)$ with model $(G,\,H)$ (with $G$ a complex connected Lie 
group and $H\,\subset\, G$ a connected complex Lie subgroup) is defined by the following data (summarized below in the 
conditions I and II).

I. A holomorphic principal $H$--bundle $\pi \,:\, E_H\,\longrightarrow\, M$ over $M$ which admits a flat partial 
holomorphic connection in the direction of $\mathcal F$: this means there exists a $H$-invariant holomorphic 
foliation $\widetilde{\mathcal F}$ of $E_H$ such that $d\pi (T \widetilde{\mathcal F})\, =\, T \mathcal F$ and the 
restriction of $d \pi$ to $\widetilde{ \mathcal F}$ is a submersion over $\mathcal F$. Hence $\widetilde{\mathcal 
F}$ is a $H$-invariant lift of $\mathcal F$ to $E_H$.

II. A holomorphic $\mathfrak g$--valued one--form $\omega\,\in\, \text{H}^0(E_H, \mathfrak g)$
 satisfying the following conditions:
\begin{enumerate}
\item $\omega$ is $H$--equivariant for the adjoint action of $H$ on $\mathfrak g$;

\item the restriction of $\omega$ to any fiber of $\pi$ coincides with the
Maurer--Cartan form $\omega_H$;

\item $\omega$ vanishes on the foliation $T{\widetilde{\mathcal F}}\, \subset\, TE_H$;

\item the induced  morphism $\omega\, :\, (TE_H)/{T\widetilde{\mathcal F}}
\, \longrightarrow\, E_H\times
{\mathfrak g}$ is constant on parallel sections of $(TE_H)/{T\widetilde{\mathcal F}}$ (with respect to the canonical connection $\nabla^{\widetilde F}$ along  the foliation $\widetilde{\mathcal F}$).

\item  the  homomorphism $\omega\, :\, (TE_H)/{T\widetilde{\mathcal  F}}
\, \longrightarrow\, E_H\times
{\mathfrak g}$ is an isomorphism.
\end{enumerate}

Condition (4) can be formulated, as in \cite{Bl}, as a vanishing of a Lie derivative. More precisely, it is 
equivalent with $L_X \omega\,=\,0$, for any $X$ tangent to $\widetilde{\mathcal F}$ (this comes from our previous 
description of flat sections of the tangent bundle of the foliation). Notice that by Lie-Cartan formula, namely 
$L_X \omega\,=\,i_X d \omega+d (i_X \omega)$, we get $L_X \omega\,=\,i_X d \omega\,=\,0$.

A {\it branched} transverse holomorphic Cartan geometry is defined by the same data, except that condition (5) is 
satisfied {\it over a nonempty open subset of $E_H$} (see \cite{BD3}). Condition (1) implies that this nonempty 
open subset of $E_H$ is the pull-back through $\pi$ of a nonempty open subset in $M$ (the complementary of a 
divisor in $M$ \cite{BD3}).
 
Condition (5) implies in both cases (classical and branched) that the complex dimension of the model $G/H$ is the 
same as the complex codimension of $\mathcal F$.

The transverse curvature $K(\omega)\,=\,d \omega + \frac{1}{2} \lbrack \omega,\, \omega \rbrack_{\mathfrak g}$ is a 
$\mathfrak g$-valued two-form on $E_H$. Conditions (1), (2), (3) and (4) imply that the curvature $K(\omega)$ 
vanishes on any pair $(Z_1, \,Z_2)\, \in\, TE_H$ such that one of them is vertical or tangent to 
$\widetilde{\mathcal F}$. Hence the transverse curvature is the pull-back to $E_H$ of a holomorphic section of 
$\Lambda^2 (\mathcal N_{\mathcal F})^* \otimes \text{ad}(E_G)$ (recall that $\text{ad}(E_G)$ is the vector bundle 
associated to $E_H$ through the action of $H \subset G$ on $\mathfrak g$ by adjoint representation).
 
When the transverse curvature vanishes identically, we have a (branched) transverse holomorphic $(G,X)$-structure. 
In this the pull-back of the (branched) flat transverse Cartan geometry on the universal cover $\widetilde M$ of 
$M$ is given by a holomorphic developing map $\text{dev} \,:\, \widetilde{M} \,\longrightarrow\, X\,=\,G/H$ which 
is constant on the leafs of $\mathcal F$. Moreover, the developing is a submersion in the classical case and a 
submersion away from the branching divisor in the branched case. The developing map is equivariant with respect to 
the monodromy homomorphism from the fundamental group of $M$ into $G$ (see \cite{BD3,Mol}).
 
In the simplest case where $G\,=\,\CC$ and $H\,=\, \{ 0 \}$, the bundle $E_H$ coincides with $M$ and a branched 
transverse structure for this model $(G,H)$ for a complex codimension one holomorphic foliation is given by a 
holomorphic one-form $\omega$ on $M$ whose kernel coincides with $\mathcal F$.  In particular, $\omega$ satisfies 
the integrability condition $\omega \wedge d \omega \,=\,0$. This is equivalent with the condition that $d \omega$ 
vanishes in restriction to $\mathcal F$.

Moreover, $\omega$ satisfies condition (4) above which is $L_X \omega\,=\, i_Xd \omega\, =\,0$, for any local holomorphic 
tangent vector field to $\mathcal F$.  Since $F$ is of complex codimension one, this implies $d \omega\,=\,0$.

Hence codimension one foliations admit branched holomorphic transverse structure exactly when they are defined as 
the kernel of a global holomorphic closed one-form $\omega$. The branching divisor is the divisor of zeros of 
$\omega$. The developing map of the transverse translation structure is a primitive of $\omega$. The monodromy 
group of the transverse translation structure is the additive subgroup of $\CC$ generated by the periods of 
$\omega$. In general this group is not a lattice.

A branched transverse holomorphic Riemannian metric is a branched transverse Cartan geometry with model $(G\,=\, 
\text{O}(n, \mathbb{C}) \ltimes {\mathbb C}^n,\, H\,=\,\text{O}(n, \mathbb{C})).$ In the flat case we get a branched 
transverse $(G,X)$-structure, with $G\,=\, \text{O}(n, \mathbb{C}) \ltimes {\mathbb C}^n$ being the group of 
complex Euclidean motions and $X$ the complex Euclidean space.
 
When the model is that of the complex affine space (respectively, that of the complex projective space) we obtain 
the notion of branched transverse holomorphic affine connection (respectively, that of branched transverse 
holomorphic projective connection). In the flat case, we get a branched transverse affine structure (respectively, 
that of a branched transverse holomorphic projective structure).
 
In the classical (unbranched) case, transversely affine (respectively transversely projective) foliations of 
complex dimension one where studied by several authors (see, for instance \cite{Sc} and references therein).
 
One could define a more general notion of a {\it generalized transverse holomorphic Cartan geometry} with model 
$(G,H)$ by dropping condition (5) (see Section \ref{section foliated Atyiah}). In this case there is no relation 
anymore between the complex codimension of $\mathcal F$ and the complex codimension of the model $G/H$. In the flat 
case the developing map of such a generalized transverse holomorphic Cartan geometry is a holomorphic map from the 
universal cover $\widetilde M$ to the model space $X\,=\,G/H$ which is constant on the leafs of $\mathcal F$. In 
general, this map is not a submersion at the generic point.
 
Notice that for the trivial foliation $\mathcal F$ (given by points in $M$) the definition of a transverse Cartan 
geometry is the same as the definition of a Cartan geometry over $M$. The same holds true in the branched case 
(respectively, in the generalized case).
 
\subsection{Foliated Atiyah bundle description}\label{section foliated Atyiah}

Consider again $\pi\, :\, TE_H\, \longrightarrow \,M$ a holomorphic (right) principal $H$--bundle over $M$. The 
kernel of the differential $d \pi$ defines a holomorphic (vertical) subbundle in the holomorphic tangent bundle 
$TE_H$ which is holomorphically isomorphic to $E_H\times {\mathfrak h}$ (this isomorphism is realized by the 
identification of the fundamental vector fields of the $H$-action with the Lie algebra of the left-invariant vector 
fields on $H$; or equivalently, using the Maurer-Cartan form $\omega_H$ of $H$). Notice that this isomorphism is 
not invariant for the lifted (right) $H$ action on $TE_H$, but equivariant with respect to the adjoint action of 
$H$ on its Lie algebra $\mathfrak h$.
 
 Following Atiyah \cite{At}, let us define the holomorphic quotient bundle over $M$ as $$ \text{ad}(E_H)\, :=\, 
\text{kernel}(\mathrm{d}\pi)/H\, \longrightarrow\, M\, . $$ As explained above, $\text{ad}(E_H)$ is holomorphically 
isomorphic with the twisted vector bundle $E_H\times^H\mathfrak h$ associated to the principal $H$--bundle $E_H$ via 
the adjoint action of $H$ on its Lie algebra $\mathfrak h$.  Recall that $\text{ad}(E_H)$ is known as the adjoint 
vector bundle of $E_H$. Notice that, since the adjoint action of $H$ on its Lie algebra preserves the Lie algebra 
structure of $\mathfrak h$, the fiber of $\text{ad}(E_H)$ has the Lie algebra structure of $\mathfrak h$: it is 
identified with $\mathfrak h$ up to a conjugation.

One can check that the quotient the quotient
$$
(TE_H)/H\, \longrightarrow\, M
$$

has also  the structure of a holomorphic vector bundle over $M$ (see \cite{At}), classically denoted by  $\text{At}(E_H)$ and  known as the {\it Atiyah bundle} of  $E_H$ \cite{At}.

The quotient by $H$ of  the short exact sequence of holomorphic vector
bundles over  $E_H$ 
$$
0\, \longrightarrow\, \text{kernel}(\mathrm{d}\pi)\, \longrightarrow\,
\mathrm{T}E_H \, \stackrel{\mathrm{d}\pi}{\longrightarrow}\,\pi^*TM
\, \longrightarrow\, 0\, .
$$

leads to the following short exact sequence of
\begin{equation}\label{at1}
0\, \longrightarrow\, \text{ad}(E_H)\, \stackrel{\iota''}{\longrightarrow}\,\text{At}(E_H)\,
\stackrel{\widehat{\mathrm{d}}\pi}{\longrightarrow}\, TM\, \longrightarrow\, 0\, ,
\end{equation}
where $\widehat{\mathrm{d}}\pi$ is constructed from $\mathrm{d}\pi$; this is known as
the Atiyah exact sequence for $E_H$. 

{\it A holomorphic connection in the principal $H$--bundle $E_H$} is a splitting of the Atiyah exact sequence \cite{At}.

Let us now consider the foliated case: the basis  $M$ of the principal $H$--bundle $E_H$  is a complex manifold endowed with a holomorphic foliation $\mathcal F$.

Define the foliated Atiyah bundle as the  subbundle
\begin{equation}\label{atF}
\text{At}_{\mathcal F}(E_H)\, :=\, (\widehat{\mathrm{d}}\pi)^{-1}(T{\mathcal F})\,\subset\,
\text{At}(E_H)\, .
\end{equation}
So from \eqref{at1} we get the short exact sequence
\begin{equation}\label{at2}
0\, \longrightarrow\, \text{ad}(E_H)\, \longrightarrow\,\text{At}_{\mathcal F}(E_H)\,
\stackrel{\mathrm{d}'\pi}{\longrightarrow}\,  T{\mathcal F}\, \longrightarrow\, 0\, ,
\end{equation}
where $\mathrm{d}'\pi  $ is the restriction of $\widehat{\mathrm{d}}\pi$
in \eqref{at1} to the subbundle $\text{At}_{\mathcal F}(E_H)$.

The principal $H$-bundle $E_H$ admits a partial holomorphic in the direction of $\mathcal F$ if the exact sequence 
(\ref{at2}) splits, meaning there exists a holomorphic homomorphism
$$
\lambda \, :\, T{\mathcal F}\, \longrightarrow\, \text{At}_{\mathcal F}(E_H)
$$
such that $\mathrm{d}'\pi\circ\lambda\,=\, \text{Id}_{T\mathcal F}$, where
$\mathrm{d}'\pi $ is the projection homomorphism in \eqref{at2}. 

Notice that the splitting of (\ref{at2}) given by the homomorphism $\lambda$ can be also  defined  using a projection   homomorphism
$p\, :\, \text{At}_{\mathcal F}(E_H)\, \longrightarrow\, \text{ad}(E_H)$ such that $p$ is the identity map in restriction  $\text{ad}(E_H)$ (the canonical  inclusion of $\text{ad}(E_H)$ in
$\text{At}_{\mathcal F}(E_H)$ is the injective homomorphism in \eqref{at2}).  The
homomorphism $p$ is  uniquely determined by  $\lambda$ (and conversely)   by the condition that
the image of $\theta$ in  $\text{At}_{\mathcal F}(E_H)$ is  the kernel of $p$.  In general  the image of $\lambda$ is not  a foliation in $TE_H$  and, as first   proved  by Ehresmann, 
the curvature of  the partial connection $\lambda \, :\, T{\mathcal F}\, \longrightarrow\, \text{At}_{\mathcal F}(E_H)$ is  the obstruction to the integrability of the image of $\lambda$.

To precise this statement, consider  two local  holomorphic sections $X_1$ and $X_2$ of $T\mathcal F$ and  compute 
the locally defined holomorphic  section $p ([\lambda(X_1),\, \lambda(X_2)])$ of $\text{ad}(E_H)$ (notice that the Lie bracket is well defined:
$\lambda(X_1)$ and $\lambda(X_2)$ representing   $H$--invariant sections  of  $TE_H$, their Lie bracket
 is also a  $H$--invariant section of $TE_H$).  One can easily check that this defines a 
${\mathcal O}_M$--linear homomorphism
$$
{\mathcal K}(\lambda) \, \in\, H^0(M,\, \text{Hom}(\bigwedge\nolimits^2{T\mathcal F},\, \text{ad}(E_H)))
\,=\, H^0(M, \, \text{ad}(E_H)\otimes \bigwedge\nolimits^2{T\mathcal F}^*)\, ,
$$
which is the \textit{curvature} of the connection $\lambda$. The connection $\lambda$ is
called {\it flat } if ${\mathcal K}(\lambda)$ vanishes identically.

A (partial) connection on the principal $H$-bundle  $E_H$ induces a canonical (partial) connection on any  bundle 
associated to $E_H$ via a representation of $H$. In particular, a (partial) connection on $E_H$ induces a 
(partial) connection on the adjoint bundle $\text{ad}(E_H)$.

Since ${\rm At}_{\mathcal F}(E_H)$ is a subbundle of ${\rm At}(E_H)$, any partial connection
$\lambda\, :\, {T\mathcal F}\, \longrightarrow\, \text{At}_{\mathcal F}(E_H)$ induces  a unique  associated  homomorphism
$\lambda'\,:\,{T\mathcal F}\, \longrightarrow\, \text{At}(E_H)$ and from \eqref{at1} we get the following  exact sequence
\begin{equation}\label{at3}
0\, \longrightarrow\, \text{ad}(E_H)\, \stackrel{\iota'}{\longrightarrow}\,
\text{At}(E_H)/\lambda'({T\mathcal F})\,
\stackrel{\widehat{\mathrm{d}}\pi}{\longrightarrow}\, TM/{T\mathcal F}\,
=\, {\mathcal N}_{\mathcal F}\, \longrightarrow\, 0\, ,
\end{equation}
where $\iota'$ is given by $\iota''$ in \eqref{at1}.

Let $\lambda$ be a flat partial connection on $E_H$. In this case the image of $\lambda$ in the foliated Atiyah 
bundle $\text{At}_{\mathcal F}(E_H)$ is a foliation. It uniquely defines a $H$-invariant foliation 
$\widetilde{\mathcal F}$ in $TE_H$ such that $d\pi (T \widetilde{\mathcal F})\, =\,T \mathcal F$ and the restriction of 
$d \pi$ to $\widetilde{ \mathcal F}$ is a submersion over $\mathcal F$. Consequently, the definition of a a flat 
partial connection on $E_H$ agrees with the one given in the context of foliated differential bundles in Section 
\ref{section foliated diff bundle}.

Recall that the normal bundle $TM/{T\mathcal F}$ is endowed with a canonical partial flat connection along 
$\mathcal F$, namely $\nabla^{\mathcal F}$, defined in Section \ref{section foliated diff bundle}.

\begin{lemma}[{\cite{BD3}}]\label{lemma1}
The  flat partial connection $\lambda$ on  $E_H$ induces  a unique  flat partial
connection on ${\rm At}(E_H)/\lambda'({T\mathcal F})$ (along $\mathcal F$)  such  that the
homomorphisms in the exact sequence \eqref{at3} are connection preserving (where $\text{ad}(E_H)$ is endowed with the canonical connection induced from $(E_H,\lambda)$ and $TM/{T\mathcal F}$ is endowed with $\nabla^{\mathcal F}$).
\end{lemma}

\begin{proof}
We have seen that the image of $\lambda$ defines an $H$--invariant holomorphic foliation $\widetilde{\mathcal F}$ 
on $E_H$; such that the differential $d\pi (T \widetilde{\mathcal F})\,=\,T \mathcal F$ and the restriction of $d \pi$ 
to $\widetilde{\mathcal F}$ is a submersion over $\mathcal F$.

Consider the canonical partial connection $\nabla^{\widetilde{\mathcal F}}$ on 
the normal bundle $TE_H/T\widetilde{\mathcal F}$. Notice that $\nabla^{\widetilde{\mathcal F}}$ is $H$-invariant.
Since $\text{At}(E_H)\,=\, (TE_H)/H$ we have $(TE_H/T \widetilde{\mathcal F})/H\,=\, {\rm 
At}(E_H)/\lambda'({\mathcal F})$. By $H$-invariance, the natural connection $\nabla^{\widetilde{\mathcal F}}$ on 
$TE_H/{T\widetilde F}$ in the direction of $\widetilde{\mathcal F}$ descends to a flat 
partial connection on ${\rm At}(E_H)/\lambda(T{\mathcal F})$ in the direction on 
$\mathcal F$.

Let us show that the morphism $\iota'$ in (\ref{at3}) is connection preserving. Consider 
$X$ a  local holomorphic vector field on $M$ tangent to $\mathcal F$, defined on an open set
$U\, \subset\, M$. Let $X'$ be the unique holomorphic vector field in $\pi^{-1}(U)\, \subset\, E_H$, tangent to
$\widetilde{\mathcal F}$, which lifts $X$, meaning that
$\mathrm{d}\pi(X')\,=\, X$. Let $N$ be a 
holomorphic section of $\text{kernel}(\mathrm{d}\pi)\, \subset\, TE_H$ over
$\pi^{-1}(U)$. Then the Lie bracket $[X', \,N]$ is such that 
$\mathrm{d}\pi([X', \,N])\,=\, 0$, meaning $[X', \,N]$ is a vertical vector field: a section of
$\text{kernel}(\mathrm{d}\pi)$. But $\text{ad}(E_H)\,=\, \text{kernel}(\mathrm{d}\pi)
/H$, and, consequently, the inclusion $\iota'$ of $\text{ad}(E_H)$ in 
${\rm At}(E_H)/\lambda'(T{\mathcal F})$ in \eqref{at3} intertwines the partial connections
on $\text{ad}(E_H)$ and ${\rm At}(E_H)/\lambda'(T{\mathcal F})$ 
in the direction of $\mathcal F$. It also follows that the projection homomorphism 
$\widehat{\mathrm{d}}\pi$ in \eqref{at3} is partial connection preserving as well.
\end{proof}

Now we give the foliated Atiyah bundle theoretical definition of a transverse branched Cartan geometry as 
introduced in \cite{BD3}.

Let $G$ be a connected complex Lie group and $H\, \subset\, G$ a complex Lie subgroup with Lie algebras $\mathfrak 
g$ and $\mathfrak h$ respectively.

Again $\pi \,:\, E_H \,\longrightarrow\, M$ is a holomorphic principal $H$--bundle over $M$.
Define 
\begin{equation}\label{eg}
E_G\,=\, E_H\times^H G\,\longrightarrow\, M
\end{equation}
be the principal $G$--bundle over  $M$ obtained by extending the structure group of $E_H$ using the 
inclusion of $H$ in $G$.  Denote by $\text{ad}(E_G)\,=\, E_G\times^G{\mathfrak g}$  the adjoint bundle for $E_G$.

The Lie algebra  inclusion of $\mathfrak h$ in $\mathfrak g$ induces an
injective homomorphism of holomorphic vector bundles 
\begin{equation}\label{i1}
\iota\, :\, \text{ad}(E_H)\,\longrightarrow\,\text{ad}(E_G)\, .
\end{equation}

Consider $\lambda$ a flat partial connection on $E_H$ in the direction of $\mathcal F$. We have seen that $\lambda$ 
induces flat partial connections on all the associated bundles, in particular, on $E_G$, $\text{ad}(E_H)$ and 
$\text{ad}(E_G)$.

A transverse  {\it branched} holomorphic Cartan geometry with model  $(G,\, H)$
on the foliated manifold $(M,\, {\mathcal F})$ is given by the following data 

I)  A holomorphic principal $H$--bundle $E_H$ on $M$ equipped with a flat partial
connection $\lambda$, and

II)  A holomorphic homomorphism

\begin{equation}\label{beta}
\beta\,:\, \text{At}(E_H)/\lambda'(T{\mathcal F})\, \longrightarrow\,
\text{ad}(E_G)\, ,
\end{equation}

such that the following three conditions hold:
\begin{enumerate}
\item $\beta$ is partial connection preserving,

\item $\beta$ is an isomorphism {\it over a nonempty open subset of $M$}, and

\item the following diagram is commutative:
\begin{equation}\label{cg1}
\begin{matrix}
0 &\longrightarrow & \text{ad}(E_H) &\stackrel{\iota'}{\longrightarrow} &
\text{At}(E_H)/\theta'(T{\mathcal F}) &
\longrightarrow & {\mathcal N}_{\mathcal F} &\longrightarrow & 0\\
&& \Vert &&~ \Big\downarrow\beta && ~ \Big\downarrow\overline{\beta}\\
0 &\longrightarrow & \text{ad}(E_H) &\stackrel{\iota}{\longrightarrow} &
\text{ad}(E_G) &\longrightarrow &
\text{ad}(E_G)/\text{ad}(E_H) &\longrightarrow & 0
\end{matrix}
\end{equation}
\end{enumerate}
where the top exact sequence is the one in \eqref{at3}, and $\iota$ is the homomorphism
in \eqref{i1}.

From the commutativity of \eqref{cg1} it follows immediately that the homomorphism $\overline{\beta} \,:\, 
{\mathcal N}_{\mathcal F}\,\longrightarrow\, \text{ad}(E_G)/\text{ad}(E_H)$ in \eqref{cg1} is an isomorphism over a 
point $m\, \in\, M$ if and only if $\beta(m)$ is an isomorphism.  Notice that the classical case of an unbranched 
transverse Cartan geometry is that where the open subset in condition II) (2) is the entire manifold $M$. This 
corresponds to condition II) (5) in the equivalent definition given in Section \ref{section foliated diff bundle}. 
Notice also that the condition II) (1) here is equivalent with the condition II (5) in Section \ref{section 
foliated diff bundle}.

Let $n$ be the complex dimension of $\mathfrak g$. Consider the homomorphism of $n$-th
exterior products
$$
\bigwedge\nolimits^n\beta\, :\, \bigwedge\nolimits^n(\text{At}(E_H)/\lambda'({T\mathcal F}))
\, \longrightarrow\, \bigwedge\nolimits^n\text{ad}(E_G)
$$
induced by $\beta$. The homomorphism $\beta$ fails to be an isomorphism precisely over the
divisor of the section $\bigwedge\nolimits^n\beta$ of the line bundle
$\text{Hom}(\bigwedge\nolimits^n(\text{At}(E_H)/\lambda'(T{\mathcal F})),
\,\bigwedge\nolimits^n\text{ad}(E_G))$. 

The {\it branching} set $D \subset M$ defined earlier  in Section \ref{section foliated diff bundle} coincides  with the vanishing set of the   holomorphic  section $\bigwedge\nolimits^n\beta$.
This divisor $\text{div}(\bigwedge\nolimits^n\beta)$ is  called the \textit{branching divisor} for $((E_H,\, \lambda),\, \beta)$.

The triple $((E_H,\, \lambda),\, \beta)$ characterizes a classical (unbranched)  holomorphic
Cartan geometry if  and only if $\beta$ is an isomorphism over  the entire $M$. In this case the branching divisor is trivial.

In order to define the more general notion of a {\it transverse generalized Cartan geometry} (which is a foliated 
version of a generalized Cartan geometry, as defined in \cite{AM,BD4}) one drops condition II) (2) in the above 
definition of a transverse branched Cartan geometry.

In order to define the developing map of a transverse (branched or generalized) holomorphic Cartan geometry the 
following results is useful.

\begin{lemma} \label{lem connection E_G}
A transverse generalized holomorphic Cartan geometry $((E_H,\, \lambda),\, \beta)$ over the foliated manifold $(M, 
\mathcal F)$ defines a holomorphic connection on the principal $G$--bundle $E_G$ which is flat in the direction of 
$\mathcal F$.
\end{lemma}

The proof below is that given in \cite{BD3}. It does not use condition II) (2) in the definition of the transverse 
Cartan geometry and works as well for transverse branched and generalized Cartan geometries.

\begin{proof} 
Consider the homomorphism
\begin{equation}\label{eh}
\text{ad}(E_H)\,\longrightarrow\, {\rm ad}(E_G)
\oplus \text{At}(E_H)\, , \ \ v \, \longmapsto\,
(\iota(v),\, -\iota''(v))
\end{equation}
(see \eqref{i1} and \eqref{at1} for $\iota$ and $\iota''$ respectively).

The corresponding quotient $({\rm ad}(E_G)\oplus \text{At}(E_H))/\text{ad}(E_H)$
is identified with the Atiyah bundle
${\rm At}(E_G)$. The inclusion of $\text{ad}(E_G)$ in ${\rm At}(E_G)$ as in
\eqref{at1} is given by the inclusion $\text{ad}(E_G) \, \hookrightarrow\,
{\rm ad}(E_G)\oplus \text{At}(E_H)$, $w\, \longmapsto\, (w,\, 0)$, while the projection
${\rm At}(E_G)\, \longrightarrow\, TM$ is given by the
composition
$$
{\rm At}(E_G)\, \hookrightarrow\, {\rm ad}(E_G)
\oplus \text{At}(E_H) \,\stackrel{(0,\widehat{\mathrm{d}}\pi)}{\longrightarrow}\, TM\, ,
$$
with $\widehat{\mathrm{d}}\pi $  the projection in \eqref{at1}.

Consider the subbundle $\lambda'(T{\mathcal F})\, \subset\, \text{At}(E_H)$ in \eqref{at3}.
The composition
$$
{\rm At}(E_H)\, \longrightarrow\, \text{At}(E_H)/\lambda'(T{\mathcal F})
\,\stackrel{\beta}{\longrightarrow}\, \text{ad}(E_G)\, ,
$$
where the first homomorphism is the quotient map, will be denoted by $\beta'$. The homomorphism
\begin{equation}\label{hv}
{\rm ad}(E_G)\oplus \text{At}(E_H)\, \longrightarrow\, {\rm ad}(E_G)\, , \ \
(v,\, w) \, \longmapsto\, v+\beta'(w)
\end{equation}
vanishes on the image of $\text{ad}(E_H)$ by the map in \eqref{eh}. Therefore,
the homomorphism in \eqref{hv} produces a homomorphism
\begin{equation}\label{vp}
\varphi\, :\, \text{At}(E_G)\,=\, ({\rm ad}(E_G)\oplus \text{At}(E_H))/\text{ad}(E_H)
\,\longrightarrow\, \text{ad}(E_G)\, .
\end{equation}
The composition
$$
\text{ad}(E_G)\,\hookrightarrow\, \text{At}(E_G)\, \stackrel{\varphi}{\longrightarrow}\,\text{ad}(E_G)
$$
 is  the identity map of $\text{ad}(E_G)$. Consequently,  $\varphi$ defines a holomorphic
connection on the principal $G$--bundle $E_G$ (see above  and also \cite{At}). 

It remains to prove  that $\varphi$ is flat along $\mathcal F$. Since the homomorphism $\beta$ in \eqref{beta} is partial connection 
preserving, it follows that the restriction of the connection $\varphi$ in the direction of $\mathcal F$ coincides with partial connection along $\mathcal F$ induced on $E_G$ by $\lambda$. Since $\lambda$ is flat along $\mathcal F$, the corresponding 
induced connection on $E_G$ is flat as well.
\end{proof} 

Denote by 
$$
{\rm Curv}(\varphi)\, \in\, H^0(M,\, \text{ad}(E_G)\otimes\Omega^2_M)
$$
the curvature of the  connection $\varphi$.

Since $\beta$ is connection preserving, the contraction of ${\rm Curv}( 
\varphi)$ with  any tangent vector of $T\mathcal F$ vanishes. This 
implies that ${\rm Curv}(\varphi)$ is actually a section of ${\rm ad}(E_G)\otimes 
\bigwedge\nolimits^2 {\mathcal N}^*_{\mathcal F}$.

The transverse (branched  or generalized) Cartan geometry $((E_H, \theta),\, \beta)$ is 
called \textit{flat} if the above  curvature tensor  ${\rm Curv}(\varphi)$ vanishes identically.

Assume that the transverse Cartan geometry is flat and $M$ is simply connected. Then the flat bundle $E_G$ is 
trivial over $M$, isomorphic to $M \times E_G$. The subbundle $E_H \subset E_G$ is described as a holomorphic 
reduction of the structure group of $E_G$ to $H$ and hence by a holomorphic map $M \longrightarrow G/H$.  This is 
the developing map $\text{dev}$ of the transverse (branched or generalized) flat Cartan geometry.
 
The differential of this developing map $\text{dev}$ is given by the homomorphism $\overline{\beta}$ in( 
\ref{cg1}). Consequently, in the branched case, the developing map is {\it a submersion away from the branching 
divisor}. There is no condition on the differential of the developing map for a transverse generalized Cartan 
geometry.
 
Since the connection $\varphi$ on $E_G$ along $\mathcal F$  is induced  from the  flat connection $\lambda$ on $E_H$, the developing map  $\text{dev}$ is constant on each (connected)  leaf of $\mathcal F$.
 
If $M$ is not simply connected the monodromy morphism of the flat connection $\varphi$ on $E_G$ is the monodromy 
morphism of the transverse branched Cartan geometry. When pulled-back on the universal cover $\widetilde M$, the 
flat bundle $E_G$ becomes trivial and the associated developing map obtained as above is equivariant with respect 
to the monodromy morphism. For more details about the construction of the developing map the reader is referred to 
\cite{Bl, Mo, BD3,BD4}.

{\it Singular Foliations.} Let us consider now a complex manifold $\widehat{M}$ endowed with a holomorphic singular 
foliation $\mathcal F$. It is classically known that there exists a maximal open dense set $M \,\subset\, \widehat{M}$ 
such that in restriction to $M$ the foliation $\mathcal F$ is a nonsingular foliation. Moreover, this maximal open 
set $M$ is the complementary of an analytic subset in $\widehat{M}$ which is of complex codimension at least two in 
$\widehat{M}$.

We will say that $(\widehat{M},\, \mathcal F)$ admits a transverse (branched or generalized) holomorphic Cartan 
geometry with model $(G,\,H)$ if $(M,\, \mathcal F)$ admits a transverse (branched or generalized) holomorphic Cartan 
geometry with model $(G,\,H)$.

An easy example of this situation is given by fibrations over homogeneous spaces. More precisely, consider a 
complex manifold $\widehat{M}$ which admits a holomorphic map $\rho$ to a complex homogeneous space $G/H$ which is 
a submersion on a nontrivial open dense set in $\widehat{M}$. Then the fibers of $\rho$ define a singular 
holomorphic foliation $\mathcal F$ on $\widehat{M}$ bearing a transverse branched holomorphic flat Cartan geometry 
with model $(G,H)$ in the sense of the above definition.

More generally, it was proved in \cite[Proposition 3.2]{BD3} that if $\rho\,:\, \widehat{M}\,\longrightarrow\, N$ 
is a holomorphic map which is a submersion of a nontrivial open dense set in $\widehat{M}$ and $N$ admits a 
holomorphic Cartan geometry with model $(G,H)$, then the holomorphic (singular) foliation $\mathcal F$ defined by 
the fibers of $\rho$ bears a transverse branched holomorphic Cartan geometry with model $(G,H)$. This transverse 
geometry is flat if and only if the Cartan geometry on $N$ is flat.
 
\subsection{Flatness results}
 
\subsubsection{Rationally connected manifolds}
 
Let us recall that a complex projective manifold $\widehat{M}$ is called {\it rationally connected} if for any pair 
of points $m,\,n \,\in\, \widehat{M}$ there exists a (maybe singular) rational curve $C\,\subset\, \widehat{M}$ such that 
$m,n \in C$.
 
Examples of rationally connected projective manifolds are given by Fano projective manifolds \cite{Ca2,KMM}. Recall 
that a projective manifold $\widehat{M}$ is Fano if its anti-canonical bundle $-K_{\widehat M}$ is ample.
 
In this context the following result was proved in \cite{BD3}.
 
\begin{theorem}\label{thmflat}
Let $\widehat M$ be a rationally connected projective manifold endowed with a (possibly singular) holomorphic 
foliation $\mathcal F$.  Then:

(1) Any transverse generalized holomorphic Cartan geometry with model $(G,\,H)$ on $(\widehat{M},\, \mathcal F)$ is 
flat and defined by a holomorphic map $\widehat{M} \,\longrightarrow\, G/H$ (constant on the leafs of $\mathcal F$);

(2) There is no transverse branched holomorphic Cartan geometry on $(\widehat{M}, \,\mathcal F)$ with model a 
nontrivial analytic affine variety $G/H$.  In particular, there is no transverse branched holomorphic affine 
connection on $(\widehat{M}, \,\mathcal F)$.
\end{theorem}

Notice that if $G$ is a complex linear algebraic group and $H$ a closed reductive 
algebraic subgroup, then $G/H$ is an affine analytic variety (see Lemma 3.32 in 
\cite{Mc}).

\begin{proof}
Let  $\widehat M$ be a  complex projective rationally connected manifolds. Consider $M \, \subset\,
\widehat{M}$ be  the maximal  open subset such that  the foliation $\mathcal F$ is nonsingular on $M$. Then 
the complex codimension of the complement of $M$ in $\widehat{M}$ is at least two.

(1) Let $((E_H,\, \lambda),\, \beta)$ be a transverse  generalized  holomorphic Cartan geometry with model 
$(G,\, H)$ on the foliated manifold $(M,\, {\mathcal F})$.

By Lemma \ref{lem connection E_G} the bundle $E_G$ (obtained from $E_H$ by extension of the structure group) 
inherits a holomorphic connection $\varphi$ over $M$ (flat in the direction of $\mathcal F$). On a smooth curve any 
holomorphic connection is flat. In particular, the restriction of $E_G$ to any smooth rational curve lying in $M$ 
is flat. Moreover, since a rational curve is simply connected, the flat holomorphic bundle $E_G$ over the rational 
curve is isomorphic to the trivial bundle endowed with the trivial connection.

There is a nonempty open subset of $M$ which is covered by smooth rational curves $C \subset M$ such that the 
restriction  $(TM)\vert_C$ of the holomorphic tangent bundle to $C$ is ample. Consequently, $H^0(C,\, 
(\Omega^2_M)\vert_C) \,=\, 0$. In particular, the curvature of the holomorphic connection $\varphi$ vanishes 
identically on $M$. By definition of its curvature, the transverse generalized holomorphic Cartan geometry 
$((E_H,\, \lambda),\, \beta)$ is flat.

Rationally connected manifolds are known to be simply connected \cite{Ca1}. Hence $\widehat{M}$ is simply 
connected. Since $M$ is a dense open subset of complex codimension at least $2$, its fundamental group is 
isomorphic to that of $\widehat{M}$. It follows that $M$ is simply connected. Therefore, the developing map 
$\text{dev}\, :\,M \, \longrightarrow\, G/H$ of the transverse generalized flat Cartan geometry $((E_H,\, \lambda 
),\, \beta)$ is defined on $M$.  Recall that the developing is constant on the leafs of $\mathcal F$.

By Hartog's extension theorem $\text{dev}$ extends as a holomorphic map defined on $\widehat{M}$. 

(2) In the branched case, the above developing map $\text{dev}\, :\, \widehat{M} \, \longrightarrow\, G/H$ is a holomorphic submersion on an open dense set and the generic leafs of $\mathcal F$ coincide with the  connected components of the fibers
of the developing map. But if $G/H$ is an analytic affine variety any holomorphic map from the compact manifold $\widehat{M}$ to $G/H$ must be constant: a contradiction. In particular, this holds if $G/H$ is the complex affine space (the model of an affine
connection transverse to $\mathcal F$).
\end{proof} 

\subsubsection{Simply connected Calabi--Yau manifolds}

Let $\widehat M$ be a simply connected compact K\"ahler manifold with vanishing real first Chern class 
$c_1(T\widehat{M})\,=\, 0$.

Those manifolds are known as simply connected {\it Calabi-Yau manifolds.} Recall that by Yau's proof of Calabi 
conjecture they are endowed with Ricci flat K\"ahler metrics.

In this context, using the main result in \cite{BD} we proved in \cite{BD3} the following result:

\begin{theorem}
Let $\widehat M$ be a simply connected Calabi-Yau manifold endowed with a (possible singular) holomorphic foliation 
$\mathcal F$. Let $G$ be a complex Lie group which is simply connected or complex semi-simple and let $H \subset G$ 
be a closed complex Lie subgroup. Then:

(1) Any transverse generalized holomorphic Cartan geometry with model $(G,\,H)$ on $(\widehat{M},\, \mathcal F)$ is 
flat and defined by a holomorphic map $\widehat{M} \,\longrightarrow\, G/H$ (constant on the leafs of $\mathcal F$).

(2) If the model $G/H$ is a nontrivial analytic affine variety $G/H$, then there is no transverse branched Cartan 
geometry on $(\widehat{M},\, \mathcal F)$ with model $(G,H)$.  In particular, there is no transverse branched 
holomorphic affine connection on $(\widehat{M},\, \mathcal F)$.
\end{theorem}

\begin{proof}
Let  $\widehat M$ be a  complex projective rationally connected manifolds. Consider $M \, \subset\,
\widehat{M}$ be  the maximal  open subset such that  the foliation $\mathcal F$ is nonsingular on $M$. Then 
the complex codimension of the complement of $M$ in $\widehat{M}$ is at least two.

(1) Let $((E_H,\, \lambda),\, \beta)$ be a transverse  generalized  holomorphic Cartan geometry with model 
$(G,\, H)$ on the foliated manifold $(M,\, {\mathcal F})$.

By Lemma \ref{lem connection E_G} the bundle $E_G$ (obtained  from $E_H$ by extension of the structure group) inherits a holomorphic connection  $\varphi$ over $M$  (flat in the direction of $\mathcal F$). 

We prove now that  $\varphi$ is flat on entire $M$.

The principal $G$--bundle $E_G$ extends to a 
holomorphic principal $G$--bundle $\widehat{E}_G$ over $\widehat M$, and the connection 
$\varphi$ extends to a holomorphic connection $\widehat\varphi$ on $\widehat{E}_G$ \cite[Theorem 
1.1]{Bi}. 

Consider  $\alpha  \,:\, G \,\longrightarrow\, {\rm GL}(N, {\mathbb C})$,  a linear representation of $G$ with discrete kernel. The corresponding Lie algebra representation $\alpha'$ is an injective Lie algebra homomorphism
from $\mathfrak g$ to $\mathfrak{ gl}(N, {\mathbb C}). $ For $G$ simply connected those representations exist by Ado's theorem. For $G$ complex 
semi-simple those representations do also exist (see Theorem 3.2, chapter XVII in \cite{H}).

Consider the holomorphic vector bundle $E_{\alpha}$  over $\widehat{M}$  with fiber type  ${\mathbb C}^N$ associated to $\widehat{E}_G$ via the representation $\alpha$.  We have seen that $E_{\alpha}$  inherits from $\widehat\varphi$  a 
holomorphic connection $\widehat{\varphi}_{\alpha}$. By Theorem \cite[Theorem 6.2]{BD}, the holomorphic  connection $\widehat{\varphi}_{\alpha}$
 is  flat over $\widehat{M}$. Since the curvature of $\widehat{\varphi}_{\alpha}$ is the image of the curvature of $\widehat{\varphi}$  through  the Lie algebra homomorphism $\alpha'$ and $\alpha'$ is injective, it follows that $\widehat{\varphi}$ is also flat. 
 Therefore $\varphi$ is flat and   the transverse  generalized holomorphic Cartan 
geometry $((E_H,\, \lambda),\, \beta)$ is flat.

Since $M$ is simply connected,  the  developing map  $\text{dev}\, :\,M \, \longrightarrow\, G/H$ of the transverse  flat generalized  Cartan geometry  $((E_H,\, \lambda ),\, \beta)$ is defined on $M$.  
Recall that  the developing is constant on the leafs of $\mathcal F$.

By Hartog's extension theorem $\text{dev}$ extends as a holomorphic map defined on $\widehat{M}$. 

(2) In the branched case, the above developing map $\text{dev}\, :\, \widehat{M} \, \longrightarrow\, G/H$ is a holomorphic submersion on an open dense set and the generic leafs of $\mathcal F$ coincide with the connected components of the fibers
of the developing map. But if $G/H$ is an analytic affine variety any holomorphic map from the compact manifold $\widehat{M}$ to $G/H$ must be constant: a contradiction. In particular, this holds if $G/H$ is the complex affine space (the model of an affine
connection transverse to $\mathcal F$).
\end{proof} 

\subsection{A topological criterion}

Let $M$ be a compact connected K\"ahler manifold of complex dimension $n$ equipped with a K\"ahler
form $\omega$. 

For any holomorphic vector bundle $V$ over $M$ we define 
\begin{equation}\label{deg}
\text{degree}(V)\,:=\, (c_1(V)\cup\omega^{n-1})\cap [M]\, \in\, {\mathbb R}\,,
\end{equation}

with $c_1(V)$ the real first Chern class of $V$.
The degree of a divisor $D$ on $M$ is defined as being  $\text{degree}({\mathcal O}_M(D))$.

The degree is a topological invariant.

Fix an effective divisor $D$ on $X$. Fix a holomorphic principal $H$--bundle $E_H$ 
on $X$.

\begin{theorem}\label{thm1}
Let $M$ be a K\"ahler manifold endowed with a nonsingular holomorphic foliation $\mathcal F$. Assume that 
$(M,\,\mathcal F)$ admits a transverse branched holomorphic Cartan geometry $((E_H,\, \lambda),\, \beta)$ with model 
$(G,H)$ and branching divisor $D \,\subset\, M$.

Then ${\rm degree}({\mathcal 
N}^*_{\mathcal F})-{\rm degree}(D)\, =\, {\rm degree}({\rm ad}(E_H))$.

In particular, if $D\, \not=\, 0$, then ${\rm degree}({\mathcal N}^*_{\mathcal F})\,>\, 
{\rm degree}({\rm ad}(E_H))$.
\end{theorem}

\begin{proof} Let $k$ be the complex dimension of the transverse model geometry $G/H$.

Recall that the homomorphism $\overline{\beta} \,:\, {\mathcal N}_{\mathcal 
F}\,\longrightarrow\, \text{ad}(E_G)/\text{ad}(E_H)$ in \eqref{cg1} is an 
isomorphism over a point $m  \in\, M$ if and only if $\beta(m)$ is an isomorphism.

The branching divisor $D$ coincides with the vanishing divisor of the holomorphic section $\bigwedge^k \overline{ \beta}$ of the holomorphic line bundle $\bigwedge^k ({\mathcal N}^*_{\mathcal F}) \otimes \bigwedge^k ({\rm ad}(E_G)/{\rm ad}(E_H))$. We have
$$
\text{degree}(D)\,=\, \text{degree}(\bigwedge\nolimits^k ({\rm ad}(E_G)/{\rm ad}(E_H))
\otimes \bigwedge^k ({\mathcal N}^*_{\mathcal F}) )
$$
\begin{equation}\label{f2}
=\, \text{degree}({\rm ad}(E_G)) - \text{degree}({\rm ad}(E_H))
+ {\rm degree}({\mathcal N}^*_{\mathcal F})\, .
\end{equation}
Recall that $E_G$ has a holomorphic connection $\phi$ (see \eqref{vp}) which induces a holomorphic connection on $\text{ad}(E_G)$. Hence we have
$c_1({\rm ad}(E_G)) \,=\, 0$ \cite[Theorem~4]{At}, which implies that
$\text{degree}({\rm ad}(E_G)) \,=\, 0$. Therefore, from \eqref{f2} it follows that
\begin{equation}\label{e7}
{\rm degree}({\mathcal N}^*_{\mathcal F})-{\rm degree}(D)\, =\, {\rm degree}({\rm ad}(E_H))\, .
\end{equation}

If $D\,\not=\, 0$, then $\text{degree}(D)\, >\, 0$. Hence in that case \eqref{e7} yields
${\rm degree}({\mathcal N}^*_{\mathcal F} )\, >\, {\rm degree}({\rm ad}(E_H))$.
\end{proof}

\begin{corollary}\label{corollaire deg}\mbox{} The hypothesis and notation  of Theorem \ref{thm1} is used.
\begin{enumerate}
\item[(i)] If ${\rm degree}({\mathcal N}^*_{\mathcal F})\, <\, 0$, then there is no transverse 
branched  holomorphic affine connection on $(M,{\mathcal F})$.

\item[(ii)] If ${\rm degree}({\mathcal N}^*_{\mathcal F})\, =\, 0$, then  every transverse  branched
 holomorphic affine connection on $(M,\mathcal F)$ has  a trivial  branching divisor on $M$.
\end{enumerate}
\end{corollary}

\begin{proof}
Recall that  the model of a transverse holomorphic affine connection on $(M,\, \mathcal F)$ is $(G,\,H)$, with 
$H\,=\,
\text{GL}(k, {\mathbb C})$ and $G\,=\, {\mathbb C}^k\rtimes\text{GL}(k, {\mathbb C})$.
The homomorphism $$\text{M}(k, {\mathbb C})\otimes
\text{M}(k, {\mathbb C})\, \longrightarrow\, \mathbb C\, ,\ \ A\otimes B\, \longmapsto\,
\text{trace}(AB)$$ is nondegenerate and $\text{GL}(k, {\mathbb C})$--invariant. In other words, 
the Lie algebra $\mathfrak h$ of $H\,=\, \text{GL}(k, {\mathbb C})$
is self-dual as an $H$--module. Hence we have $\text{ad}(E_H)\,=\, \text{ad}(E_H)^*$, in particular,
the equality $$\text{degree}(\text{ad}(E_H))\,=\,0$$ holds. Hence from
Theorem  \ref{thm1},
\begin{equation}\label{prc}
{\rm degree}({\mathcal N}^*_{\mathcal F})\,=\, {\rm degree}(D)\, .
\end{equation}

For  the effective divisor $D$ we have ${\rm degree}(D)\, \geq \, 0$. Moreover, for a {\it nonzero} effective divisor $D$ we have $\text{degree}(D)\, >\, 0$.
Therefore, the two points of the  corollary follow from \eqref{prc} and Theorem \ref{thm1}.
\end{proof}

\section{Some related open problems}

We present here some open questions dealing with holomorphic $G$-structures and holomorphic (foliated) Cartan 
geometries of compact complex manifolds.

\textbf{$\text{SL}(2,{\mathbb C})$--structures in odd dimension and holomorphic Riemannian metrics}

We have seen that any holomorphic $\text{SL}(2,{\mathbb C})$--structure on a complex manifold of odd dimension $M$ produces an associated  {\it holomorphic Riemannian metric } on $M$. Moreover, for the complex dimension three, these two structures are equivalent.
 
Exotic compact complex threefolds endowed with holomorphic Riemannian metrics (or equivalently, $\text{SL}(2,{\mathbb C})$--structures) were constructed by  Ghys  in \cite{Gh} using  deformations of quotients of 
$\text{SL}(2,{\mathbb C})$ by normal lattices. Moreover, Ghys proved in \cite{Gh} that all  holomorphic Riemannian metrics on those   exotic deformations  are  {\it locally homogeneous} (meaning that local holomorphic vector fields preserving the
holomorphic Riemannian metric on $M$ span the holomorphic tangent bundle $TM$).

It was proved in 
\cite{Du} that all holomorphic Riemannian metrics on compact complex threefolds  are locally 
homogeneous. This means that  any compact complex threefold  $M$ bearing a holomorphic Riemannian metric  admits a flat holomorphic Cartan geometry, or equivalently, a holomorphic $(G,X)$-structure  such that the holomorphic Riemannian metric on $M$ comes form
a global $G$-invariant holomorphic Riemannian metric on the model space $X$. The classification of all possible models $(G,X)$ was done in \cite{DZ} where it was deduced that all compact complex threefolds endowed with a (locally homogeneous)  holomorphic Riemannian metric 
also admit a finite unramified cover equipped with a holomorphic Riemannian metric of constant sectional curvature.
 
We conjecture that {\it ${\rm SL}(2,{\mathbb C})$--structures on compact complex manifolds of odd 
dimension are always locally homogeneous.} More generally we conjectured that {\it holomorphic Riemannian metrics on compact complex manifolds are locally homogeneous}. 

Of course, these conjectures generalize to the non K\"ahler framework, known results in the K\"ahler context (see  Theorem \ref{sl2} and also the  discussion below  about the Fujiki case).

Some evidence toward these conjectures  is given by the recent result  in \cite{BD2} proving that simply connected compact complex manifolds do not admit  holomorphic Riemannian metrics. Notice that in the locally homogeneous case, this 
would be a direct  consequence of the fact that the developing map $\text{dev}: M \longrightarrow X$ should be a submersion  from a compact manifold to a complex affine model: a contradiction.

It should be clarified  that the above mentioned conjectures are an important {\it  partial} step toward the classification of compact complex manifolds with holomorphic Riemannian metrics. Indeed, even    the classification of compact complex manifolds endowed
with {\it flat}  holomorphic Riemannian metrics is still an open problem. To understand the flat case one should prove first that $M$ is a quotient of the complex Euclidean space $(\CC^n, dz_1^2+ \ldots +dz_n^2)$  by a discrete subgroup of  the group of complex Euclidean motions 
$O(n, \CC) \ltimes \CC^n$ (this is a special holomorphic  version  of Markus conjecture
which asserts that compact manifolds  with unimodular affine structures, meaning here  locally modeled on $(G,X)$, with $X=\mathbb R^{2n}$  and  $G=SL(2n, \mathbb R) \ltimes \mathbb R^{2n}$ are complete:  they are quotients of the model space $X$ by a discrete subgroup of $G$.

The second step would be to classify discrete subgroups of $G=O(n, \CC) \ltimes \CC^n$ acting properly discontinuously and with a compact quotient on the model  $(\CC^n, dz_1^2+ \ldots +dz_n^2)$. Notice that a  special case of Auslander conjecture predicts that those subgroups are
virtually solvable (i.e., they admit a finite index subgroup which is solvable). More precisely, the general Auslander conjecture  states
 that {\it a  flat complete  unimodular affine compact manifold has virtually solvable fundamental group.} Notice that  compactness is essential in the statement of Auslander conjecture; non compact flat complete unimodular affine manifolds  with non abelian free fundamental group were constructed by Margulis.

\textbf{Holomorphic affine connections  on Fujiki class $\mathcal C$ manifolds}

Recall that a compact K\"ahler manifold bearing a holomorphic connection in its holomorphic tangent bundle $TM$ has 
vanishing real Chern classes \cite{At}. We have seen that this result can be obtain following Chern-Weil theory and 
computing real Chern classes of the holomorphic tangent bundle $TM$ first using a hermitian metric on it and then 
the holomorphic connection. This provides representatives of the Chern class $c_k(TM) \in \text{H}^{2k}(M, \mathbb 
R)$ which are smooth forms on $M$ of type $(k,k)$ (given by the first computation) and other of type $(2k,0)$ (when 
computed via the holomorphic connection). Classical Hodge theory says that on K\"ahler manifolds nontrivial real 
cohomology classes do not have representatives of different types. This implies the vanishing of the real Chern 
class: $c_k(TM)=0,$ for all $k$. This part of the argument directly adapts to compact complex manifolds manifolds 
which are bimeromorphic to K\"ahler manifolds; those manifolds were studied in \cite{Fu} and are now said to be 
{\it in the Fujiki class $\mathcal C$.} All images of compact K\"ahler manifolds through holomorphic maps are known 
to belong to the Fujiki class $\mathcal C$ \cite{Va}.

We have also seen that a compact K\"ahler manifold $M$ with vanishing first two Chern classes ($c_1(TM)=c_2(TM)=0$) 
admits a finite unramified covering by a compact complex torus \cite{IKO,Ya}.
 
We conjecture that the same is true for Fujiki class $\mathcal C$ manifolds, namely {\it a compact complex manifold 
in the Fujiki class $\mathcal C$ having vanishing first two Chern classes admits a finite unramified covering by a 
compact complex torus.}

In particular, we conjecture that {\it a complex compact manifold $M$ in the Fujiki class $\mathcal C$ bearing a 
holomorphic connection in its holomorphic tangent bundle $TM$ admits a finite unramified covering by a compact 
complex torus.} This was proved in \cite[Proposition 4.2]{BD5} for {\it Moishezon manifolds} (i.e. compact complex 
manifolds bimeromorphic to projective manifolds) \cite{Mo}. Also this was proved in \cite[Theorem C]{BDG} for 
holomorphic Riemannian metrics on Fujiki class $\mathcal C$ manifolds under some technical assumption (which could 
probably be removed).

\textbf{$\text{GL}(2,{\mathbb C})$--structures on compact K\"ahler manifolds of odd dimension}
 
Recall that   complex manifolds of odd dimension  bearing  a $\text{GL}(2,{\mathbb C})$--structure also admit a holomorphic conformal structure. Flat conformal structures on 
compact projective manifolds were classified in \cite{JR2}: beside some projective surfaces, there are only 
the standard examples.
 
As for the conclusion in Theorem \ref{Fano} and in Theorem \ref{KE}, we conjecture  that {\it K\"ahler 
manifolds of odd complex dimension $\geq 5$ and bearing a holomorphic $\text{GL}(2,{\mathbb 
C})$--structure are covered by compact tori.}

\textbf{Holomorphic Cartan geometries on  compact complex surfaces}

In several papers  Inoue, Kobayashi and Ochiai  studied holomorphic affine and projective connections on compact complex surfaces. A consequence of their work is  that compact complex surfaces bearing holomorphic affine 
connections (respectively, holomorphic projective connections) also admit {\it flat } holomorphic affine connections (respectively, flat holomorphic 
projective 
connections) with corresponding injective developing map. In particular, those complex surfaces are uniformized as quotients of open subsets in the 
complex affine plane (respectively, complex projective plane) by a discrete subgroup of affine transformations (respectively, projective 
transformations)  
acting 
properly and discontinuously \cite{IKO, KO1,KO2}. 

We conjecture that those results can be  generalized to all holomorphic Cartan geometries; namely, that compact complex surfaces bearing  holomorphic Cartan geometries with model $(G,\,H)$ also admit {\it flat}  holomorphic
Cartan geometries with model $(G,\,H)$ and with corresponding injective developing map into $G/H$. This would  provide uniformizations result for compact complex surfaces bearing holomorphic Cartan geometries with model $(G,\,H)$ as compact quotients of open subsets $U$ in $G/H$ by discrete subgroups in $G$ preserving $U \subset G/H$ and acting properly and discontinuously on $U$.

In order to address this problem locally (in the neighborhood of a Cartan geometry in the deformation space), the deformation theory  for holomorphic  (non necessarily flat)  Cartan geometries was recently  worked out in \cite{BDS}.

One could  also naturally ask the analogous  question   in the framework of {\it branched}  holomorphic Cartan geometries on compact complex surfaces. Recall that the branched framework is much broader since all projective surfaces admit
 branched flat holomorphic projective connections \cite{BD}. Moreover,  branched torsion free  holomorphic affine connections on compact complex surfaces which are {\it non  projectively flat } were constructed in \cite{BD}: these branched affine connections are 
 essential (meaning they are not obtained as pull-back of unbranched ones by a ramified holomorphic map) since it is known that (unbranched) torsion free  holomorphic affine connections on compact complex surfaces are {\it projectively flat} \cite{Du2}.

\textbf{Foliated Cartan geometries on compact complex tori}

This question is about holomorphic foliations on complex torii with transverse holomorphic Cartan geometry.

Our motivation for this question came  from Ghys  classification of   codimension one  holomorphic  nonsingular foliations on   complex tori \cite{Gh2} which be briefly describe here.  

I)  The simplest examples of codimension one  holomorphic foliations are those given by the kernel of some holomorphic 1-form  $\omega$.  Since holomorphic one-forms on complex tori are  necessarily translation invariant, the foliation
given by the kernel of $\omega$  is  also translation invariant.

II) Assume now that $T=\CC^{n}/ \Lambda$ , with $\Lambda$ a lattice in  $\CC^{n}$ such that  there exists a linear form $L : \CC^{n} \longrightarrow \CC$ sending $\Lambda$  to a lattice $\Lambda'$ in $\CC$.  Then $L$  descends
to a map $\widehat{L}  : T \longrightarrow  \CC/ \Lambda'$. Pick a nonconstant meromorphic function  $u$  on the elliptic curve $\CC/ \Lambda'$ and consider the meromorphic closed 1-form $\Omega=\widehat{L}^{*}(udz) + \omega$  on 
$T$, with $\omega$ as above and $dz$ a uniformizing  holomorphic 1-form on $\CC/ \Lambda'$ . It is easy to check  that the foliation given by the kernel of $\Omega$ extends to all of $T$ as a nonsingular  holomorphic codimension one foliation. This foliation is not invariant by all translations in $T$, but only by those which lie the kernel of $L$. 
They act on $T$ as a subtorus of symmetries of codimension one.

Ghys's theorem asserts that all  codimension one (nonsingular)  holomorphic  foliations on complex tori are constructed as in  I) or II) above. In particular, they are invariant by a subtorus of complex codimension one.
Moreover, for generic complex torii, there are no nonconstant meromorphic functions and, consequently the construction II) does not apply. All holomorphic codimension one foliations on generic  tori  are translation invariant.

We recently proved in \cite{BD6}  that all holomorphic  Cartan geometries  on complex tori are translation invariant.

We conjecture that the foliated analogous also holds, namely that {\it all holomorphic  foliations bearing  transverse holomorphic Cartan geometries on complex tori are translation invariant.}

\section*{Acknowledgements}

We warmly thank Charles Boubel who kindly explained to us the geometric construction  of the noncompact dual $D_3$ of the   three dimensional quadric $Q_3$ presented in Section \ref{section G-struct}.
This work has been supported by the French government through the UCAJEDI Investments in the 
Future project managed by the National Research Agency (ANR) with the reference number 
ANR2152IDEX201. The first-named author is partially supported by a J. C. Bose Fellowship, and 
school of mathematics, TIFR, is supported by 12-R$\&$D-TFR-5.01-0500. 


\end{document}